\documentclass[11pt]{article}
\usepackage{amssymb,amsmath,amsfonts,amsthm,mathrsfs}
\usepackage{graphicx,graphics,psfrag,epsfig,calrsfs,cancel}
\usepackage[colorlinks=true, pdfstartview=FitV, linkcolor=blue, citecolor=red, urlcolor=blue]{hyperref}

\usepackage{float}
\usepackage{caption,subfig,color}
\usepackage{calrsfs}
\usepackage[dvipsnames]{xcolor}
\usepackage{tikz}
\usetikzlibrary{calc}

\usepackage{geometry}
\newtheorem{definition}{Definition}
\newtheorem{theorem}{Theorem}

\newtheorem{proposition}{Proposition}
\newtheorem{corollary}{Corollary}
\newtheorem{lemma}{Lemma}

\newcommand{\N}{\mathbb{N}}
\newcommand{\Z}{\mathbb{Z}}
\newcommand{\R}{\mathbb{R}}

\begin{document}

\title{On some stable boundary closures \\ of finite difference schemes for the transport equation}

\author{Jean-Fran\c{c}ois {\sc Coulombel}\thanks{Institut de Math\'ematiques de Toulouse ; 
UMR5219, Universit\'e de Toulouse ; CNRS, Universit\'e Paul Sabatier, F-31062 Toulouse Cedex 9, France. 
Email: {\tt jean-francois.coulombel@math.univ-toulouse.fr}. Research of the author was supported by ANR project NABUCO, ANR-17-CE40-0025.} 
$\,$ \& Tomas {\sc Lundquist}\thanks{Institut de Math\'ematiques de Toulouse ; 
UMR5219, Universit\'e de Toulouse ; CNRS, Universit\'e Paul Sabatier, F-31062 Toulouse Cedex 9, France. 
Email: {\tt tomas.lundquist@math.univ-toulouse.fr}. Research of the author was supported by ANR project NABUCO, ANR-17-CE40-0025.}}
\date{\today}
\maketitle

\begin{abstract}
We explore in this article the possibilities and limitations of the so-called energy method for analyzing the stability of finite difference approximations 
to the transport equation with extrapolation numerical boundary conditions at the outflow boundary. We first show that for the most simple schemes, 
namely the explicit schemes with a three point stencil, the energy method can be applied for proving stability estimates when the scheme is implemented 
with either the first or second order extrapolation boundary condition. We then examine the case of five point stencils and give several examples of 
schemes and second order extrapolation numerical boundary conditions for which the energy method produces stability estimates. However, we 
also show that for the standard first or second order translatory extrapolation boundary conditions, the energy method cannot be applied for proving 
stability of the classical fourth order scheme originally proposed by Strang. This gives a clear limitation of the energy method with respect to the more 
general approach based on the normal mode decomposition.
\end{abstract}



\section{Introduction}

A general approach for studying the stability of numerical boundary conditions for discretized hyperbolic equations has been initiated in the fundamental 
contribution \cite{gks}. However, for technical reasons, the stability estimates in \cite{gks} are restricted to zero initial data and can be obtained only after 
verifying the fulfillment of some ``algebraic'' condition which is commonly referred to as the Uniform Kreiss-Lopatinskii Condition. The theory in \cite{gks} 
has been successfully applied to some well-known extrapolation procedures at outflow boundaries, see a preliminary announcement in \cite{kreissproc} 
and the complete proof in \cite{goldberg}. In the recent work \cite{CL}, one of the authors has revisited the stability estimates for the outflow extrapolation 
procedures in \cite{kreissproc,goldberg} and shown that a suitably devised energy argument could bypass the (technical and lenghty) arguments of 
\cite{gks}. In this article, we examine in a systematic way the possible applications but also the limitations of the energy method for analyzing the stability 
of extrapolation procedures at an outflow boundary for the most simple one-dimensional transport equation. We begin by reviewing some central concepts 
and notation that will be extensively used throughout the remainder of this paper.

\subsection{Continuous problem}

We consider the scalar advection equation in $1D$ on a semi-infinite interval with an outflow boundary at $x=0$,
\begin{equation}\label{eq:cont}
\begin{aligned}
u_t+u_x &= 0 , \quad -\infty \leq x\leq 0 \, ,\quad t \ge 0 \, .
\end{aligned}
\end{equation}
The differential operator $\partial_t+\partial_x$ is \textit{semi-bounded} on this domain, i.e. for all $0 \le t_1 \le t_2$ and integrable 
smooth functions $\Phi(x,t)$ where $\Phi(-\infty,t)=0$, we have,
\begin{equation}\label{eq:cont_en}
\Phi_t + \Phi_x=0 \quad \Rightarrow \quad \| \Phi \|^2(t_2) \, = \, \| \Phi \|^2(t_1) - \int_{t_1}^{t_2} \Phi(0,\tau)^2 \, {\rm d}\tau \, \leq \, \| \Phi \|^2(t_1) \, ,
\end{equation}
where we have used the standard definition of the $L_2$ inner product and norm (with respect to the space variable $x$) for real valued functions,
\begin{equation*}
(\Phi,\Psi) \, := \, \int_{-\infty}^0 \Phi(x) \, \Psi(x) \, {\rm d}x \, ,\quad \|\Phi\|^2 \, := \,  (\Phi,\Phi) \, .
\end{equation*}
When later discretizing (\ref{eq:cont}) in time and space, we shall seek to mimic the energy balance (\ref{eq:cont_en}).

\subsection{Discretizing on the whole real line}

We seek to discretize the problem (\ref{eq:cont}) in time and space using a finite difference method. In the interior of the spatial domain we use a repeated 
interior stencil, which we analyze for the Cauchy version of (\ref{eq:cont}). We introduce an equispaced spatial grid over the whole real line, 
\begin{equation*}
x_j \, := \, j \, \Delta x, \quad -\infty \leq j \leq \infty \, ,
\end{equation*}
where $\Delta x>0$ is the grid spacing. Unless otherwise stated, sequences ($\Phi_j$) on the grid are assumed to be real valued, and the norm on $\ell^2 (\Z)$ 
is defined by,
$$
\| \Phi \|_{\ell^2(\mathbb{Z})}^2 \, := \, \sum_{j \in \Z} \, \Phi_j^2 \, .
$$

Next, we introduce a time step $\Delta t>0$ (and accordingly, $t_n=n\Delta t$) and consider the ratio $z:=\Delta t/\Delta x$ (the so-called Courant-Friedrichs-Lewy 
parameter \cite{cfl}) as a constant. In what follows, $u_j^n$ stands for an approximation of the solution $u$ to the Cauchy version of \eqref{eq:cont} in the neighborhood 
of $(x_j,t_n)$. We consider explicit finite difference discretizations to (\ref{eq:cont}) of the form:
\begin{equation}\label{eq:scheme}
u_j^{n+1} \, = \, A(z) \, u_j^{n} \, ,\quad -\infty \leq j \leq \infty \, ,\quad n \in \mathbb{N} \, ,
\end{equation}
where $A(z)$ is a polynomial with respect to $z=\Delta t/\Delta x$ of difference operators. E.g, decomposing $A(z)$ in terms of simple shift operators, $A(z)$ then acts 
on sequences $(\Phi_j)$ indexed by $j$ according to,
$$
A(z) \, \Phi_j \, := \, \sum_{\ell=-\ell_-}^{\ell_+} a_\ell (z) \, \Phi_{j+\ell} \, ,
$$
where the integers $\ell_\pm$ mark the extent (or bandwidth) of the difference stencil, and where the coefficients $a_\ell (z)$ are polynomial expressions with respect to $z$. 
Most of the examples we will consider below fall into the classes of three point stencils ($\ell_-= \ell_+ =1$) and five point stencils ($\ell_-=\ell_+=2$). It is clear that $A(z)$ acts 
boundedly on $\ell^2(\mathbb{Z})$, and we shall say that $A(z)$ is a contraction if the norm of $A(z)$ as an operator on $\ell^2(\mathbb{Z})$ is not larger than $1$.

In what follows, we only consider finite difference schemes that are consistent with the transport equation \eqref{eq:cont}. In other words, see for instance \cite{RM,gko}, 
we only consider operators $A(z)$ in \eqref{eq:scheme} that satisfy at least the first order accuracy relations,
$$
\sum_{\ell=-\ell_-}^{\ell_+} a_\ell (z) \, = \, 1 \, ,\quad \text{\rm and } \quad \sum_{\ell=-\ell_-}^{\ell_+} \ell \, a_\ell (z) \, = \, -z \, .
$$

In order to prove stability estimates, be they on the whole space $\mathbb{Z}$ or on the half-space $\mathbb{Z}^-$, we aim to use the energy method in order to prove 
that the scheme is a contraction with respect to a certain norm. For the Cauchy problem, we thus need for $A(z)$ to be a contraction on $\ell^2(\mathbb{Z})$. Since the 
coefficients of the operator $A(z)$ do not depend on $j$, we can achieve this by splitting each local energy balance term into a telescopic and a dissipative part. In other 
words, the goal will be to show that, for an arbitrary sequence $(\Phi_j)_{j \in \mathbb{Z}}$,
\begin{equation}\label{eq:quad_tele}
\Psi_j \, = \, A(z) \, \Phi_j \quad \Rightarrow \quad \Psi_j^2 \, = \, \Phi_j^2 \, + \, T_j-T_{j-1} +S_j \, , \quad S_j \le 0 \, , 
\end{equation}
for all values of $z$ within a certain specified interval. In \eqref{eq:quad_tele}, $T_j$, $T_{j-1}$ and $S_j$ are quadratic quantities over the $\Phi_k$ entries for $j-\ell_- \le k 
\le j+\ell_+$. When summing over all whole numbers, the telescopic part cancels and we are left with
\begin{equation*}
\| \Psi \|_{\ell^2(\mathbb{Z})}^2  \, = \,\| \Phi \|_{\ell^2(\mathbb{Z})}^2+ \sum_{j \in \mathbb{Z}} S_j \, \le \, \| \Phi \|_{\ell^2(\mathbb{Z})}^2 \, ,
\end{equation*}
showing that $A(z)$ is a contraction, just as the continuous problem (\ref{eq:cont_en}).

The main goal of analyzing schemes of the form (\ref{eq:scheme}) will be to derive \emph{integration by parts decompositions} of the form \eqref{eq:quad_tele} with 
$S_j \le 0$ for any $j$. There is a lot of freedom involved in splitting the difference $\Psi_j^2 - \Phi_j^2$ into telescopic and non-telescopic parts as in (\ref{eq:quad_tele}), 
see for instance \cite[Lemma A.1]{CL}. In this paper we will take advantage of this in order to expand and refine the energy method arguments used in \cite{CL}.

\subsubsection{Finite difference operators}

The basic unit for constructing all finite difference stencils will be the backward difference operator,
\begin{equation}
\label{eq:back_diff}
D \, \Phi_j \, := \, \Phi_j -\Phi_{j-1} \, .
\end{equation}
In order to define a parametric family of schemes with a three point stencil, we will use the standard first and second order derivative approximations (the "centered finite 
difference" and the "discrete Laplacian", respectively). On normalized form with respect to the spacing $\Delta x$, we thus consider,
\begin{align}
\label{eq:central_diff}
 D_0 \, \Phi_j \,  = &  \, \dfrac{1}{2} \, \left( D \, \Phi_j +D \, \Phi_{j+1} \right) \, = \, \dfrac{1}{2} \, \left( \Phi_{j+1}-\Phi_{j-1} \right) \, ,\\
 \label{eq:laplace_diff}
\Delta \, \Phi_j \, = & \, D \, \Phi_{j+1} -D \, \Phi_j \, = \, \Phi_{j-1} -2 \, \Phi_j +\Phi_{j+1} \, .
\end{align}
Note that both of these operators can easily be shown to be of accuracy order $2$. To define a family of five point stencils, we also need to approximate the third and fourth 
derivatives, which we define using,
\begin{align}
\label{eq:3rd_der}
D_0 \, \Delta \, \Phi_j \, = & \, \dfrac{1}{2} \, \left( -\Phi_{j-2} +2 \, \Phi_{j-1} -2 \, \Phi_{j+1} +\Phi_{j+2} \right) \, ,\\
\label{eq:4th_der}
\Delta^2 \, \Phi_j \, =  & \, \Phi_{j-2} -4 \, \Phi_{j-1} +6 \, \Phi_j -4 \, \Phi_{j+1} +\Phi_{j+2} \, .
\end{align}
Since we will not consider schemes with more than a five point stencil, in this paper we will limit our attention to only the finite difference operators introduced above.

\subsection{Outflow boundary conditions}

We close the discrete domain at the outflow boundary $x=0$, and define a general vector on the grid as,
\begin{equation*}
\mathbf{\Phi} \, = \, \begin{pmatrix}
\ldots, \Phi_{-2},\Phi_{-1} , \Phi_0 \end{pmatrix} \, .
\end{equation*}
At the points nearest to the outflow boundary at $x=0$, extrapolation to ghost points outside the computational domain are needed in order to close the finite difference 
stencils and thus the numerical scheme (\ref{eq:scheme}) itself. Namely, we need a method to define the $\Phi_1,\dots,\Phi_{\ell_+}$ values in order to compute $\Psi_j 
:=A(z) \, \Phi_j$ for all $j \in \mathbb{Z}^-$. Note that this extrapolation step may or may not reduce the local accuracy order from that of the interior stencil itself.

Next, we define an inner product on $\mathbb{Z}^-$ in the form of a quadrature rule, such that in general only the last $r$ quadrature weights $h_j>0$ are non-unity,
\begin{equation}\label{eq:inner_prod}
\|\boldsymbol{\Phi}\|_H^2 \, := \, (\boldsymbol{\Phi},\boldsymbol{\Phi})_H \, ,\quad 
(\boldsymbol{\Phi},\boldsymbol{\Psi})_H \, := \, \sum_{j\leq -r} \Phi_j \, \Psi_j +\sum_{j= -r+1}^{0} h_j \, \Phi_j \, \Psi_j \, .
\end{equation}
Now let's assume that contractivity (\ref{eq:quad_tele}) holds for all sequences $(\Phi_j) \in \ell^2(\mathbb{Z})$. Then, using an arbitrary set of 
extrapolation conditions to define the $\Phi_j$ values with positive indices $j=1,\dots,\ell_+$, we have,
\begin{equation}\label{defE}
\forall \, j \in \mathbb{Z}^- \, ,\quad \Psi_j \, = \, A(z) \, \Phi_j \quad \Rightarrow  \quad 
\|\boldsymbol{\Psi}\|_H^2 -\|\boldsymbol{\Phi}\|_H^2 -\sum_{j \leq -r} S_j \, = \, \sum_{j=-r}^0 (h_j-h_{j+1}) \, T_j +\sum_{j=-r+1}^0 h_j \, S_j \, =: \, E \, ,
\end{equation}
where we have defined $h_{-r}:=1$ and $h_1:=0$. Since the local dissipation $S_j$ is nonpositive for any $j$, we get from \eqref{defE}:
$$
\forall \, j \in \mathbb{Z}^- \, ,\quad \Psi_j \, = \, A(z) \, \Phi_j \quad \Rightarrow  \quad 
\|\boldsymbol{\Psi}\|_H^2 \, \le \,\|\boldsymbol{\Phi}\|_H^2+ E \, .
$$
The goal at this point is to determine the extrapolation procedure and the quadrature weights such that $E$ is, at least, non-positive. In that case, the combination 
of the operator $A(z)$ with the extrapolation procedure yields a contraction on $\ell^2(\mathbb{Z}^-)$ provided that the latter space is equipped with the norm 
(\ref{eq:inner_prod}). We formalize this notion below.

\begin{definition}\label{def:semi_b_stab}
Together with suitable extrapolation conditions at an outflow boundary, we say that the operator $A(z)$ in (\ref{eq:scheme}) is semi-bounded with respect 
to the inner product defined in (\ref{eq:inner_prod}) if, for a given $z\geq 0$, the following two conditions are satified.
\begin{enumerate}
\item The scheme is contractive, that is $A(z)$ admits a decomposition (\ref{eq:quad_tele}) with $S_j \le 0$.
\item The boundary contribution is non-positive, i.e.
\begin{equation}\label{eq:semi_b_stab}
E \, := \, \sum_{j =-r}^0 (h_j-h_{j+1}) \, T_j +\sum_{j=-r+1}^0 h_j \, S_j \leq 0 \, ,
\end{equation}
where we use the convention $h_{-r} :=1$ and $h_1 :=0$.
\end{enumerate}
\end{definition}

As a matter of fact, it is useful for the purpose of convergence proofs (see, e.g., \cite{gustafsson,gko,CL}) to infer from \eqref{defE} a trace estimate on the sequence 
$\boldsymbol{\Phi}$. This will be the case provided that $E$ is a negative definite quadratic form of its arguments, as evidenced below on several examples. 
This energy argument bypasses, as in \cite{wu,jfcag,CL}, the lengthy verification of the fulfillment of the so-called Uniform Lopatinskii Condition. We refer to 
\cite{gks,jfcnotes} for a general presentation of this topic.

The rest of this article is organized as follows. In Section \ref{sect:3pts}, we examine the application of the energy method described above in the case of 
three point schemes ($\ell_-=\ell_+=1$). Since such schemes have at most second order accuracy \cite{gko}, we restrict our analysis to the case of first and 
second order extrapolation at the outflow boundary. We then examine the case of five point schemes ($\ell_-=\ell_+=2$) in Section \ref{sect:5pts}. Since the 
involved algebra becomes heavier, we restrict sometimes the analysis to three particular cases which are analogues of some well-known three point schemes. 
We give examples of extrapolation procedures for which we can construct a suitable energy that yields semi-boundedness. In contrast to what is more commonly 
used, these energy stable extrapolation procedures are not translation invariant. Finally, in section 5 we demonstrate the limitation of more standard extrapolation 
procedures, proving that for a fourth order five point scheme combined with the most common first and second order extrapolation procedures, the energy method 
actually fails to work, though the corresponding closure of the operator $A(z)$ is known to be power-bounded. This shows that the normal mode analysis developed 
in \cite{gks} and subsequent works is sometimes necessary to capture stability in general.

\section{Three point schemes}
\label{sect:3pts}

The family of three point stencils of at least first order accuracy can be defined using a single free parameter $\nu$ as follows (see \cite{RM,gko}),
\begin{equation}
\label{defAz3pts}
A(z) \, = \, I -z \, D_0+\dfrac{\nu}{2} \, \Delta \, .
\end{equation}
Famous examples include the Euler forward ($\nu = 0$), Lax-Friedrichs ($\nu = 1$), the two-point upwind ($\nu = z$) as well as the second order accurate 
Lax-Wendroff $(\nu = z^2)$ schemes.

\subsection{Contractivity on the whole real line}

For the family of schemes considered in \eqref{defAz3pts}, our first result is the following.

\begin{lemma}
\label{lem1}
Let $z,\nu \in \R$, and consider the scheme (\ref{eq:scheme}) with $A(z)$ as in \eqref{defAz3pts}. Then (\ref{eq:quad_tele}) holds with
\begin{equation}\label{eq:tele_threepoint}
T_j \, := \, \begin{pmatrix}
\Phi_j \\ D \, \Phi_{j+1} \end{pmatrix}^T Q \, \begin{pmatrix}
\Phi_j \\ D \, \Phi_{j+1} \end{pmatrix} \, , \quad Q \, := \, \begin{pmatrix}
-z & \dfrac{\nu-z}{2} \\
\dfrac{\nu-z}{2} & \dfrac{\nu \, (1-z)}{2} \end{pmatrix} \, ,
\end{equation}
and
\begin{equation}\label{eq:symm_threepoint}
S_j \, = \, \begin{pmatrix}
D \, \Phi_j \\ D \, \Phi_{j+1} \end{pmatrix}^T M \, \begin{pmatrix}
D \, \Phi_j \\ D \, \Phi_{j+1} \end{pmatrix} \, ,\quad M \, := \, a_1 \, \begin{pmatrix}
1 & 0 \\ 0 & 1 \end{pmatrix} +a_2 \, \begin{pmatrix}
1 & -1 \\ -1 & 1 \end{pmatrix} \, ,
\end{equation}
where
\begin{equation*}
a_1 \, := \, -\dfrac{\nu-z^2}{2} \, ,\quad a_2 \, := \, \dfrac{\nu^2-z^2}{4} \, .
\end{equation*}
In particular, $A(z)$ is a contraction on $\ell^2(\mathbb{Z})$, i.e. $S_j\le 0$, if $z$ and $\nu$ satisfy the relation $z^2 \le \nu \le 1$.

\begin{proof}
With $A(z)$ as in \eqref{defAz3pts}, we have in (\ref{eq:quad_tele}),
\begin{equation}\label{eq:int_en_threepoint}
\Psi_j^2 - \Phi_j^2 \, = \, -2 \, z \, \Phi_j \, (D_0 \, \Phi_j) +z^2 \, (D_0 \, \Phi_j)^2 +\nu \, \Phi_j \, (\Delta \, \Phi_j) -z \, \nu \, (D_0 \, \Phi_j) \, (\Delta \, \Phi_j) 
+\dfrac{\nu^2}{4} \, (\Delta \, \Phi_j)^2 \, .
\end{equation}
In order to split this expression into telescopic and non-telescopic parts, we first note that the $(D_0 \, \Phi_j)^2$ term 
can be expressed as the sum of squares (this is nothing but the parallelogram identity), 
\begin{equation}
\label{eq:tele_01}
(D_0 \, \Phi_j)^2 \, = \, \dfrac{1}{4} \, \left[ 2 \, (D \, \Phi_j)^2 +2 \, (D \, \Phi_{j+1})^2 -(\Delta \, \Phi_j)^2 \right] \, .
\end{equation}
Moreover, we will use the purely telescopic formulas,
\begin{align}
\label{eq:tele_02}
2 \, \Phi_j \, D_0 \, \Phi_j & \, = \, \Phi_j \, \Phi_{j+1} -\Phi_{j-1} \, \Phi_j \, = \, 
\left( \Phi_j^2 +\Phi_j \, D \, \Phi_{j+1} \right) -\left( \Phi_{j-1}^2 +\Phi_{j-1} \, D \, \Phi_j \right) \, ,\\
\label{eq:tele_03}
2 \, D_0 \, \Phi_j \, \Delta \, \Phi_j & \, = \, (D \, \Phi_{j+1})^2 -(D \, \Phi_j)^2 \, .
\end{align}
Finally, the $\Phi_j \, \Delta \, \Phi_j$ term can be written as the combination of non-telescopic (sum of squares) and telescopic parts\footnote{The decomposition 
mimics the equality $2 \, u \, u'' =-2\, (u')^2 +(u^2)''$.},
\begin{align}
\notag 
2 \, \Phi_j \, \Delta \, \Phi_j & \, = \, -(D \, \Phi_j)^2 -(D \, \Phi_{j+1})^2 +\Phi_{j+1}^2 -2 \, \Phi_j^2 +\Phi_{j-1}^2 \\
\label{eq:tele_04}
& \, = \, -(D \, \Phi_j)^2 -(D \, \Phi_{j+1})^2 +\left[ 2 \, \Phi_j \, D \, \Phi_{j+1} +(D \, \Phi_{j+1})^2 \right] -\left[ 2 \, \Phi_{j-1} \, D \, \Phi_j +(D \, \Phi_j)^2 \right] \, .
\end{align}
Inserting the formulas (\ref{eq:tele_01}) through (\ref{eq:tele_04}) into (\ref{eq:int_en_threepoint}), we get the desired form $\Psi_j^2 -\Phi_j^2=T_j-T_{j-1}+S_j$ 
with $T_j$ directly as in \eqref{eq:tele_threepoint} and $S_j$ given by
\begin{equation*}
S_j \, = \, a_1 \, \left[ (D \, \Phi_j)^2 +(D \, \Phi_{j+1})^2 \right] +a_2 \, (\Delta \, \Phi_j)^2 \, .
\end{equation*}
From the definition of the discrete Laplacian $\Delta$ in (\ref{eq:laplace_diff}), we can further rewrite this into (\ref{eq:symm_threepoint}). This proves the first part 
of the Lemma.

Next, an orthogonal set of eigenvectors to $M$ is given by $(1 \ 1)^T$ and $(1 \ -1)^T$, and the associated 
eigenvalues are,
\begin{equation*}
\lambda_1 \, = \, a_1 \, = \, -\dfrac{\nu-z^2}{2} \, ,\quad \lambda_2 \, = \, a_1+2 \, a_2  \, = \, -\dfrac{\nu \, (1-\nu)}{2} \, .
\end{equation*}
We have,
\begin{equation*}
\begin{aligned}
\lambda_1 \leq 0 & \quad \Leftrightarrow & z^2 & \leq \nu \, ,\\
\lambda_2 \leq 0 & \quad \Leftrightarrow & \nu & \in [0,1] \, .
\end{aligned}
\end{equation*}
and we have thus shown that the scheme defined by $A(z)$ is a contraction if $z^2 \leq \nu \leq 1$. 
\end{proof}
\end{lemma}

It is worth to mention here that the same exact bounds on $z,\nu$ for $\ell^2$ stability given in Lemma \ref{lem1} above can also be derived using Fourier analysis 
(they are both necessary and sufficient for $A(z)$ to be a contraction). However, for a wider stencil we can not in general guarantee that the energy method will yield 
the same bounds as obtained with the Fourier method. Since we will use the energy method to derive extrapolation boundary conditions leading to semi-boundedness, 
we shall in this paper only consider stencils for which contractivity can be shown using the energy method in a similar fashion to the one above.

To conclude this section, we write down the eigenvalues associated with the quadratic form $S_j$ for the important special cases mentioned 
at the beginning of this section. The Euler forward method ($\nu = 0$) yields,
\begin{equation*}
\lambda_1 \, = \, \dfrac{z^2}{2} \, ,\quad \lambda_2 \, = \, 0 \, ,
\end{equation*}
showing that this scheme is not contractive for any $z$. Moreover, the Lax-Friedrichs scheme ($\nu = 1$) yields,
\begin{equation*}
\lambda_1 \, = \, -\dfrac{1-z^2}{2} \, ,\quad \lambda_2 \, = \, 0 \, ,
\end{equation*}
the upwind scheme ($\nu = z$) yields,
\begin{equation*}
\lambda_1 \, = \, \lambda_2 \, = \, -\dfrac{z \, (1-z)}{2} \, ,
\end{equation*} 
and finally the Lax-Wendroff $(\nu = z^2)$ scheme,
\begin{equation*}
\lambda_1 \, = \, 0 \, ,\quad \lambda_2 \, = \, -\dfrac{z^2 \, (1-z^2)}{4} \, .
\end{equation*}
This shows that the Lax-Friedrichs and Lax-Wendroff schemes are both contractive for all $|z| \leq 1$ (meaning that we can freely change the sign in front 
of $u_x$ in (\ref{eq:cont}) without modifying the stability property of the scheme), while the upwind scheme is contractive for $0\leq z \leq 1$.

\subsection{Outflow boundary conditions}

For simplicity, in \eqref{eq:inner_prod} we let $r=1$ from the start, i.e. the norm of $\boldsymbol{\Phi} \in \ell^2(\mathbb{Z}^-)$ in (\ref{eq:inner_prod}) is defined by
\begin{equation*}
\|\boldsymbol{\Phi}\|_H^2 \, = \, \sum_{j \leq -1} \Phi_j^2 +h_0 \, \Phi_0^2 \, ,
\end{equation*}
and for now we consider the last quadrature weight $h_0$ as a free (positive) parameter. In (\ref{eq:semi_b_stab}) we then have, using Lemma \ref{lem1} 
from the previous paragraph,
\begin{align}
E \, =& \, (1-h_0) \, T_{-1} +h_0 \, T_0 +h_0 \, S_0 \notag \\
=& \, (1-h_0) \, \begin{pmatrix}
\Phi_{-1} \\ D \, \Phi_0 \end{pmatrix}^T \, Q \, \begin{pmatrix}
\Phi_{-1} \\ D \, \Phi_0 \end{pmatrix} +h_0 \, \begin{pmatrix}
\Phi_0 \\ D \, \Phi_1 \end{pmatrix}^T \, Q \, \begin{pmatrix}
\Phi_0 \\ D \, \Phi_1 \end{pmatrix} +h_0 \, \begin{pmatrix}
D \, \Phi_0 \\ D \, \Phi_1 \end{pmatrix}^T \, M \, \begin{pmatrix}
D \, \Phi_0 \\ D \, \Phi_1 \end{pmatrix} \, .\label{expressE}
\end{align}
Inserting \eqref{eq:tele_threepoint} and \eqref{eq:symm_threepoint} into the above expression \eqref{expressE} for $E$ gives us right away a stability estimate 
for the most popular first order extrapolation boundary closure.

\begin{corollary}
\label{cor1}
Let $z \in (0,1]$ and let $\nu$ satisfy $z^2 \le \nu \le 1$. Then the scheme (\ref{defAz3pts}) together with the first order extrapolation condition $\Phi_1=\Phi_0$ is 
semi-bounded. In particular, the energy balance is given by,
\begin{equation}
\label{energie3ptsneumann11}
\sum_{j \le 0} \Psi_j^2 \le \sum_{j \le 0} \Phi_j^2 -z \, \Phi_0^2  \le  \, \sum_{j \le 0}  \Phi_j^2 \,  ,
\end{equation}
which exactly mimics (\ref{eq:cont_en}) with added dissipation.
\end{corollary}

\begin{proof}
We choose $h_0=1$ for the last quadrature weight. From the extrapolation condition we can use $D \, \Phi_1=0$ in (\ref{expressE}) to immediately get,
$$
E \, = \, \, \begin{pmatrix}
\Phi_{0} \\ 0 \end{pmatrix}^T \, Q \, \begin{pmatrix}
\Phi_{0} \\ 0 \end{pmatrix} +\begin{pmatrix}
D \, \Phi_0 \\ 0 \end{pmatrix}^T \, M \, \begin{pmatrix}
D \, \Phi_0 \\ 0 \end{pmatrix} \, = \, -z \, \Phi_0^2 +S_0 \, .
$$
The energy balance \eqref{energie3ptsneumann11} then follows from \eqref{defE} and the fact that $S_j \le 0$ for any $j \le 0$ if $z^2 \le \nu \le 1$.
\end{proof}

We now examine the case of a second order extrapolation condition at the outflow boundary, that is we impose $\Delta \, \Phi_0 \, = \, 0$. Our result is the following.

\begin{corollary}
\label{cor2}
Let $z \in (0,1]$ and let $\nu$ satisfy $z^2 \le \nu \le 1$, and consider the scheme (\ref{defAz3pts}) together with the second order extrapolation condition 
$\Delta \, \Phi_0=0$. Then $E$ in (\ref{eq:semi_b_stab}) is given by
\begin{equation}
\label{energie3pts-2}
E \, = \, \begin{pmatrix}
\Phi_0 \\ D \, \Phi_0 \end{pmatrix}^T \, B \, \begin{pmatrix}
\Phi_0 \\ D \, \Phi_0 \end{pmatrix} \, ,\quad B \, = \, \begin{pmatrix}
-z & z(\frac{1}{2}-h_0)+\frac{\nu}{2} \\
z(\frac{1}{2}-h_0)+\frac{\nu}{2} & h_0 z^2-\frac{\nu(1+z)}{2}
\end{pmatrix} \, .
\end{equation}
Moreover, $B$ is negative definite (and thus the scheme is semi-bounded) for the choice $h_0 = (1-z+\nu/z)/2$.

\begin{proof}
From the second order extrapolation condition, we have,
\begin{equation*}
\Delta \, \Phi_0 \, = \, 0 \quad \Rightarrow \quad D \, \Phi_1 \, = \, D \, \Phi_0 \, .
\end{equation*}
We also have, from the definition of $D$ in (\ref{eq:back_diff}),
\begin{equation*}
\Phi_{-1} \, = \, \Phi_0 -D \, \Phi_0 \, .
\end{equation*}
Going back to \eqref{expressE}, we now write $E$ as,
\begin{equation*}
E \, = \, (1-h_0) \, \begin{pmatrix}
\Phi_0-D \, \Phi_0 \\ D \, \Phi_0 \end{pmatrix}^T \, Q \, \begin{pmatrix}
\Phi_0-D \, \Phi_0 \\ D \, \Phi_0 \end{pmatrix} 
+h_0 \, \begin{pmatrix}
\Phi_0 \\ D \, \Phi_0 \end{pmatrix}^T \, Q \, \begin{pmatrix}
\Phi_0 \\ D \, \Phi_0 \end{pmatrix} 
+h_0 \, \begin{pmatrix}
D \, \Phi_0 \\ D \, \Phi_0 \end{pmatrix}^T \, M \, \begin{pmatrix}
D \, \Phi_0 \\ D \, \Phi_0 \end{pmatrix} \, ,
\end{equation*}
which, using (\ref{eq:tele_threepoint}) and (\ref{eq:symm_threepoint}), can be further simplified into (\ref{energie3pts-2}).

In order to see whether $B$  in (\ref{energie3pts-2}) is negative definite, we rotate to diagonal form using the similarity transformation, 
\begin{equation*}
R \, := \, \begin{pmatrix}
1 & \left( \dfrac{\nu+z}{2} -z \, h_0 \right) /z \\
0 & 1 \end{pmatrix} \quad \Rightarrow  \quad R ^T B \, R \, = \, \begin{pmatrix}
-z & 0 \\
0 & \gamma \end{pmatrix} \, ,
\end{equation*}
where
\begin{equation*}
\gamma \, := \, \left[ \left( \nu +z -2 \, z \, h_0 \right)^2 +4 \, h_0 \, z^3 -2 \, z \, (1+z) \, \nu \right] /(4 \, z) \, .
\end{equation*}
The minimum of $\gamma$ with respect to $h_0$ is obtained for
\begin{equation*}
\dfrac{\partial \gamma}{\partial h_0} \, = \, z^2 -z -\nu +2 \, h_0 \, z \, = \, 0 \quad \Leftrightarrow \quad h_0 \, = \, \dfrac{\frac{\nu}{z}-z+1}{2} \ge \dfrac{1}{2} \, ,
\end{equation*}
which in all cases leads to the same negative value of $\gamma$,
\begin{equation*}
\gamma \, = \, -z^3/4 \, .
\end{equation*}
With this choice of $h_0$, the quantity \eqref{eq:inner_prod} (with $r=1$) defines a norm on $\ell^2(-\infty,0)$ and the second condition of Definition 
\ref{def:semi_b_stab} is always satisfied for $z \in (0,1]$ and $\nu \in [z^2,1]$. Semi-boundedness is thus obtained as long as the interior scheme is 
contractive. 
\end{proof}
\end{corollary}

In particular, the second order accurate Lax-Wendroff scheme ($\nu=z^2$) is associated with the likewise second order accurate Trapezoidal rule with 
the constant value $h_0 = 1/2$ (independently of $z \in (0,1]$). For all other three-point schemes, the above value for $h_0$ as obtained with this method 
depends on $z$ and grows to infinity as $z$ approaches $0$ ($\nu$ being fixed).

\section{Five point schemes}
\label{sect:5pts}

In this preliminary work we do not aim for a complete theory encompassing finite difference schemes of any order. Instead we go on to consider a slightly 
more involved example than in the previous section, namely the family of schemes with a five point stencil and at least second order accuracy. Such schemes 
can be parametrized with two free parameters $\sigma,\tau$ as follows,
\begin{equation}\label{defAz5pts}
A(z) \, = \, I -z \, D_0 +\dfrac{z^2}{2} \, \Delta +\sigma \, D_0 \, \Delta +\tau \, \Delta^2 \, .
\end{equation}
We first derive a general integration by parts decomposition formula for $A(z)$ and thereby obtain three interesting examples for which $A(z)$ is contractive.

\subsection{Contractivity on the whole real line}

Generalizing Lemma \ref{lem1} to the family of operators in \eqref{defAz5pts}, our result is the following.

\begin{lemma}
\label{lem2}
Let $z,\sigma,\tau \in \R$, and consider the scheme (\ref{eq:scheme}) with $A(z)$ as in \eqref{defAz5pts}. Then (\ref{eq:quad_tele}) holds with
\begin{multline}\label{eq:tele_fiveepoint}
T_j \, := \, \begin{pmatrix}
\Phi_j \\ D \, \Phi_{j+1} \\ \Delta \, \Phi_j \\ D \, \Delta \, \Phi_{j+1} \end{pmatrix} ^T Q \, \begin{pmatrix}
\Phi_j \\ D \, \Phi_{j+1} \\  \Delta \, \Phi_j \\ D \, \Delta \, \Phi_{j+1} \end{pmatrix} \, , \\
Q \, := \, \begin{pmatrix}
-z & -\frac{z \, (1-z)}{2} & \sigma & \frac{\sigma}{2} +\tau \\ 
-\frac{z \, (1-z)}{2} & \frac{z^2 \, (1-z)}{2} -\sigma & \frac{\sigma}{2} -z \, \sigma - \tau & -z \, \left( \frac{\sigma}{2} +\tau \right) \\
\sigma & \frac{\sigma}{2} -z \, \sigma -\tau & \frac{z^2 \, \sigma}{2} +z \, \tau & \frac{z \, (1+z)}{2} \, \left( \frac{\sigma}{2} +\tau \right) -\frac{a_1}{3} \\
\frac{\sigma}{2}+\tau & -z \, \left( \frac{\sigma}{2}+\tau \right) & \frac{z \, (1+z)}{2} \, \left( \frac{\sigma}{2} +\tau \right) -\frac{a_1}{3} & 
\frac{z \, \sigma}{4} +\frac{z^2 \, \tau}{2} +\sigma \, \tau -\frac{a_1}{3} \end{pmatrix} \, ,
\end{multline}
and
\begin{equation}\label{eq:symm_fiveepoint}
S_j := \begin{pmatrix}
\Delta \, \Phi_{j-1} \\ \Delta \, \Phi_j \\ \Delta \, \Phi_{j+1}
\end{pmatrix}^TM\begin{pmatrix}
\Delta \, \Phi_{j-1} \\ \Delta \, \Phi_j \\ \Delta \, \Phi_{j+1}
\end{pmatrix} \, , \, M \, := \, \dfrac{a_1}{3} \, \begin{pmatrix}
1 & 0 & 0 \\ 0 & 1 & 0 \\ 0 & 0  & 1
\end{pmatrix} +a_2 \, \begin{pmatrix}
1 & -1 & 0 \\
-1 & 2 & -1 \\
0 & -1 & 1
\end{pmatrix} +a_3 \, \begin{pmatrix}
1 & -2 & 1 \\
-2 & 4 & -2 \\
1 & -2 & 1 
\end{pmatrix} \, ,
\end{equation}
where
\begin{equation*}
a_1 \, := \, \dfrac{z^2 \, (z^2-1)}{4} +2 \, z \, \sigma +2 \, \tau \, ,\quad a_2 \, := \, -\dfrac{z \, \sigma}{4} +\dfrac{\sigma^2}{2} -\dfrac{z^2 \, \tau}{2} \, ,\quad 
 a_3 \, := \, -\dfrac{\sigma^2}{4} +\tau^2 \, .
\end{equation*}
In particular, $A(z)$ is a contraction on $\ell^2(\mathbb{Z})$ in each of the three following cases:
\begin{itemize}
 \item $\sigma=0$, $\tau=-(1-z^2)/12$ and $z \in [-1,1]$,
 \item $\sigma=z \, (1-z^2)/6$, $\tau=-\sigma/2$ and $z \in [0,1]$,
 \item $\sigma=z \, (1-z^2)/6$, $\tau=-z^2 \, (1-z^2)/24$ and $z \in [-1,1]$.
\end{itemize}
\end{lemma}

\begin{proof}
Using \eqref{defAz5pts} in (\ref{eq:quad_tele}), we now have,
\begin{align}
\nonumber
\Psi_j^2 - \Phi_j^2 \, = \, &-2 \, z \, \Phi_j \, D_0 \, \Phi_j +z^2 \, (D_0 \, \Phi_j)^2 +z^2 \Phi_j \, \Delta \, \Phi_j -z^3 \, (D_0 \, \Phi_j) \, (\Delta \, \Phi_j) 
+\dfrac{z^4}{4} \, (\Delta \, \Phi_j)^2 
\\
\nonumber
& +2 \, \sigma \, \Phi_j \, D_0 \, \Delta \, \Phi_j - 2 \, z \, \sigma \, (D_0 \Phi_j) \, (D_0 \, \Delta \, \Phi_j) +z^2 \, \sigma \, (\Delta \, \Phi_j) \, (D_0 \, \Delta \, \Phi_j) 
+\sigma^2 \, (D_0 \, \Delta \, \Phi_j)^2 
\\
\nonumber
& +2 \, \tau \, \Phi_j \, (\Delta^2 \, \Phi_j) -2 \, z \, \tau \, (D_0 \, \Phi_j) \, (\Delta ^2 \, \Phi_j) +z^2 \, \tau \, (\Delta \, \Phi_j) \, (\Delta^2 \, \Phi_j) 
+2 \, \sigma \, \tau \, (D_0 \, \Delta \, \Phi_j) \, (\Delta ^2 \, \Phi_j) 
\\
\label{eq:int_en_fivepoint}
& +\tau^2 \, (\Delta ^2 \, \Phi_j)^2 \, .
\end{align}
The terms in the first line on the right hand side above are covered by the previous formulas (\ref{eq:tele_01}), (\ref{eq:tele_02}), (\ref{eq:tele_03}) 
and (\ref{eq:tele_04}). Further, by substituting $\Phi_j$ for $\Delta \, \Phi_j$ in the same formulas, we get,
\begin{equation}
\label{eq:tele_01_fivepoint}
(D_0 \, \Delta \, \Phi_j)^2 \, = \, \dfrac{1}{4} \, \left[ 2 \, (D \, \Delta \, \Phi_j)^2 +2 \, (D \, \Delta \, \Phi_{j+1})^2 -(\Delta^2 \, \Phi_j)^2 \right] \, ,
\end{equation}
as well as
\begin{align}
\label{eq:tele_02_fivepoint}
2 \, (\Delta \, \Phi_j) \, (D_0 \, \Delta \, \Phi_j) & \, = \, \left[ (\Delta \, \Phi_j)^2 +\Delta \, \Phi_j \, D \, \Delta \, \Phi_{j+1} \right] 
-\left[ (\Delta \, \Phi_{j-1})^2 +\Delta \, \Phi_{j-1} \, D \, \Delta \, \Phi_j \right] \, ,\\
\label{eq:tele_03_fivepoint}
2 \, (D_0 \, \Delta \, \Phi_j) \, (\Delta^2 \, \Phi_j) & \, = \, (D \, \Delta \, \Phi_{j+1})^2 -(D \, \Delta \, \Phi_j)^2 \, ,
\end{align}
and 
\begin{equation}
\label{eq:tele_04_fivepoint}
\begin{aligned}
2 \, \Delta \, \Phi_j \, \Delta^2 \, \Phi_j \, = \, & -(D \, \Delta \, \Phi_j)^2 -(D \, \Delta \, \Phi_{j+1})^2 \\
& +\left[ 2 \, \Delta \, \Phi_j \, D \, \Delta \, \Phi_{j+1} +(D \, \Delta \, \Phi_{j+1})^2 \right] -\left[ 2 \, \Delta \, \Phi_{j-1} \, D \, \Delta \, \Phi_j +(D \, \Delta \, \Phi_j)^2 \right] \, .
 \end{aligned}
\end{equation}
For the remaining terms on the right hand side of \eqref{eq:int_en_fivepoint}, we need to introduce the following four new high order formulas, the proof 
of which can be found in Appendix,
\begin{multline}\label{eq:tele_11}
2 \, D_0 \, \Phi_j \, D_0 \, \Delta \, \Phi_j \, = \, -2 \, (\Delta \, \Phi_j)^2 +\dfrac{1}{4} \, \left[ (D \, \Delta \, \Phi_j)^2 +(D \, \Delta \, \Phi_{j+1})^2 \right] \\
+\left[ (D \, \Phi_{j+1} -\dfrac{1}{4} \, D \, \Delta \, \Phi_{j+1}) \, (D \, \Delta \, \Phi_{j+1} +2 \, \Delta \, \Phi_j) \right] 
-\left[ (D \, \Phi_j -\dfrac{1}{4} \, D \, \Delta \, \Phi_j) \, (D \, \Delta \, \Phi_j +2 \, \Delta \, \Phi_{j-1}) \right] \, ,
\end{multline}
as well as,
\begin{align}
\label{eq:tele_12}
2 \, \Phi_j \, D_0 \, \Delta \, \Phi_j = & \, \left[ \Phi_j \, (D \, \Delta \, \Phi_{j+1} +2 \, \Delta \, \Phi_j) +(\Delta \, \Phi_j -D \, \Phi_{j+1}) \, D \, \Phi_{j+1} \right] \\
\nonumber
& \, - \left[ \Phi_{j-1} \, (D \, \Delta \, \Phi_j +2 \, \Delta \, \Phi_{j-1}) +(\Delta \, \Phi_{j-1} -D \, \Phi_j) \, D \, \Phi_j \right] \, ,\\
\label{eq:tele_13}
2 \, D_0 \, \Phi_j \, \Delta^2 \, \Phi_j = & \, \left[ -(D \, \Delta \, \Phi_{j+1} +\Delta \, \Phi_j) \, \Delta \, \Phi_j +2 \, D \, \Phi_{j+1} \, D \, \Delta \, \Phi_{j+1} \right] \\
\nonumber
 & - \left[ -(D \, \Delta \, \Phi_j +\Delta \, \Phi_{j-1}) \, \Delta \, \Phi_{j-1} +2 \, D \, \Phi_j \, D \, \Delta \, \Phi_j \right] \, ,
\end{align}
and 
\begin{equation}
\label{eq:tele_14}
2 \, \Phi_j \, \Delta^2 \, \Phi_j \, = \, 2 \, (\Delta \, \Phi_j)^2 + \left[ 2 \, \Phi_j \, D \, \Delta \, \Phi_{j+1} -2 \, D \, \Phi_{j+1} \, \Delta \, \Phi_j \right] 
-\left[ 2 \, \Phi_{j-1} \, D \, \Delta \, \Phi_j -2 \, D \, \Phi_j \, \Delta \, \Phi_{j-1} \right] \, .
\end{equation}

Note that all the non-telescopic parts above are expressed in terms of the three variables $(\Delta \, \Phi_j)^2$, $(\Delta^2 \, \Phi_j)^2$ and $(D \, \Delta \, \Phi_j)^2 
+(D \, \Delta \, \Phi_{j+1})^2$. The last two of these can be written as,
\begin{equation*}
(\Delta^2 \, \Phi_j)^2 \, = \, \begin{pmatrix}
\Delta \, \Phi_{j-1} \\ \Delta \, \Phi_j \\ \Delta \, \Phi_{j+1}
\end{pmatrix}^T \, \begin{pmatrix}
1 \\ -2 \\ 1 \end{pmatrix} \, \begin{pmatrix}
1 & -2 & 1 \end{pmatrix} \, \begin{pmatrix}
\Delta \, \Phi_{j-1} \\ \Delta \, \Phi_j \\ \Delta \, \Phi_{j+1}
\end{pmatrix} \, = \, \begin{pmatrix}
\Delta \, \Phi_{j-1} \\ \Delta \, \Phi_j \\ \Delta \, \Phi_{j+1} 
\end{pmatrix}^T \, \begin{pmatrix}
1 & -2 & 1 \\
-2 & 4 & -2 \\
1 & -2 & 1 
\end{pmatrix} \, \begin{pmatrix}
\Delta \, \Phi_{j-1} \\ \Delta \, \Phi_j \\ \Delta \, \Phi_{j+1}
\end{pmatrix} \, ,
\end{equation*}
and
\begin{equation*}
(D \, \Delta \, \Phi_j)^2 +(D \, \Delta \, \Phi_{j+1})^2 \, = \, \begin{pmatrix}
\Delta \, \Phi_{j-1} \\ \Delta \, \Phi_j \\ \Delta \, \Phi_{j+1}
\end{pmatrix}^T \, \begin{pmatrix}
1 & -1 & 0 \\
-1 & 2 & -1 \\
0 & -1 & 1 
\end{pmatrix} \, \begin{pmatrix}
\Delta \, \Phi_{j-1} \\ \Delta \, \Phi_j \\ \Delta \, \Phi_{j+1}
\end{pmatrix} \, .
\end{equation*}
Thus, in order to analyze the sign of the complete non-telescopic part, it is convenient to expand $(\Delta \, \Phi_j)^2$ into a similar matrix form with a telescopic correction,
\begin{multline*}
(\Delta \, \Phi_j)^2 \, = \, \dfrac{1}{3} \, \begin{pmatrix}
\Delta \, \Phi_{j-1} \\ \Delta \, \Phi_j \\ \Delta \, \Phi_{j+1} \end{pmatrix}^T \, \begin{pmatrix}
1 & 0 & 0 \\
0 & 1 & 0 \\
0 & 0 & 1 
\end{pmatrix} \, \begin{pmatrix}
\Delta \, \Phi_{j-1} \\ \Delta \, \Phi_j \\ \Delta \, \Phi_{j+1} \end{pmatrix} \\
-\dfrac{1}{3} \, \left[ (D\Delta \, \Phi_j+\Delta \, \Phi_j)^2 +(\Delta \, \Phi_j)^2 \right] 
+\dfrac{1}{3} \, \left[ (D\Delta \, \Phi_{j-1}+\Delta \, \Phi_{j-1})^2+(\Delta \, \Phi_{j-1})^2 \right] \, .
\end{multline*}
Inserting (\ref{eq:tele_01_fivepoint}) through (\ref{eq:tele_14}) into (\ref{eq:int_en_fivepoint}) and then rewriting $(\Delta \, \Phi_j)^2$, $(\Delta^2 \, \Phi_j)^2$ and 
$(D \, \Delta \, \Phi_j)^2 +(D \, \Delta \, \Phi_{j+1})^2$ according to the above, the first part of the Lemma follows. Note that the term $a_1 \, (\Delta \, \Phi_j)^2$ has 
been split between between the telescopic and dissipative parts.

The same orthogonal set of eigenvectors to each of the three terms in (\ref{eq:symm_fiveepoint}) is given by $(1,1,1)^T$, $(-1,0,1)^T$ and $(1,-2,1)^T$, with the 
associated eigenvalues of $M$,
\begin{equation*}
\lambda_1 \, := \, \dfrac{a_1}{3} \, ,\quad \lambda_2 \, := \, \dfrac{a_1}{3} +a_2 \, ,\quad \lambda_3 \, := \, \dfrac{a_1}{3} +3 \, a_2 +6 \, a_3 \, . 
\end{equation*}
The operator $A(z)$ is a contraction on $\ell^2(\mathbb{Z})$ if all three eigenvalues $\lambda_{1,2,3}$ of $M$ are non-positive. In contrast to the three 
point stencil case, we can not easily rewrite the conditions for contractivity into explicit relations for $\sigma$ and $\tau$. However, three special cases 
seem to be worth to mention. First, if we let $\sigma=0$ and $\tau = -(1-z^2)/12$, then the eigenvalues of $M$ simplify into,
\begin{equation*}
\lambda_1 \, = \, -\dfrac{(1-z^2) \, (2 +3 \, z^2)}{36} \, ,\quad \lambda_2 \, = \, -\dfrac{(1-z^2) \, (4 +3 \, z^2)}{72} \, ,\quad 
\lambda_3 \, = \, -\dfrac{1-z^2}{72} \, .
\end{equation*}
Hence $A(z)$ is a contraction for any $z \in [-1,1]$ in that case. Recall that for the second order Lax-Wendroff scheme, the dissipation approaches 
$0$ as $z$ approaches $0$, which is not the case here. In terms of dissipation properties, this new scheme can instead be seen in some sense to 
be analogous to the three-point Lax-Friedrichs scheme (though one eigenvalue for the Lax-Friedrichs scheme is zero, which is not the case here).

A skewed $4-$point stencil of accuracy order $3$, which can be seen as analogous to the two-point upwind scheme, is given by $\sigma := z \, (1-z^2)/6$ 
and $\tau := -\sigma/2=-z(1-z^2)/12$, yielding,
\begin{align*}
\lambda_1 \, = \, -\dfrac{z \, (1-z^2) \, (2-z)}{36} \, ,\quad \lambda_2 \, =& \, -\dfrac{z \, (1-z^2) \, (1+z) \, (2-z)^2}{72} \, , \\
\lambda_3 \, =& \, -\dfrac{z \, (1-z^2) \, (2-z) \, (2 +3 \, z -3 \, z^2)}{72} \, .
\end{align*}
Again, $A(z)$ is a contraction on $\ell^2(\mathbb{Z})$ for any $z \in [0,1]$. Finally, the so-called Strang scheme of fourth order accuracy \cite{strang} 
is given by $\sigma := z \, (1-z^2)/6$ and $\tau := -z^2 \, (1-z^2)/24$, with the corresponding eigenvalues,
\begin{equation*}
\lambda_1 \, = \, 0 \, ,\quad \lambda_2 \, = \, -\dfrac{z^2 \, (1-z^2) \, (4-z^2)}{144} \, ,\quad 
\lambda_3 \, = \, -\dfrac{z^2 \, (1-z^2) \, (3-z^2) \, (4-z^2)}{96} \, .
\end{equation*}
The operator $A(z)$ is again a contraction on $\ell^2(\mathbb{Z})$ for any $z \in [-1,1]$. This scheme is a high (i.e., fourth) order analogue to the second 
order accurate Lax-Wendroff scheme.
\end{proof}
\noindent 
We note that the corresponding stability analysis for the fourth order Strang scheme in \cite{strang} is performed by means of the Fourier transform and 
is therefore not applicable for the outflow problem, which will be considered next.

\subsection{Outflow boundary conditions}

We examine in this paragraph two sets of outflow boundary conditions based on second order accurate extrapolation, and thus leave higher order boundary 
conditions as a topic for future work. As opposed to the three point stencil case, we have not found a single set of second order boundary conditions which is 
stable for all of the three example schemes listed in Lemma \ref{lem2}, which is why we consider two different alternatives below.

In what follows, we shall use the inner product \eqref{eq:inner_prod} with $r=1$ and $h_0=1/2$. In (\ref{eq:semi_b_stab}), we thus have,
\begin{equation}\label{eq:E5pts_0}
E \, = \, \dfrac{1}{2} \, 
\begin{pmatrix}
\Phi_{-1} \\ D \, \Phi_0 \\ \Delta \, \Phi_{-1} \\ D \, \Delta \, \Phi_0 \end{pmatrix}^T \, Q \, \begin{pmatrix}
\Phi_{-1} \\ D \, \Phi_0 \\ \Delta \, \Phi_{-1} \\ D \, \Delta \, \Phi_0 \end{pmatrix} 
+\dfrac{1}{2} \, \begin{pmatrix}
\Phi_0 \\ D \, \Phi_1 \\ \Delta \, \Phi_0 \\ D \, \Delta \, \Phi_1 \end{pmatrix}^T \, Q \, \begin{pmatrix}
\Phi_0 \\ D \, \Phi_1 \\ \Delta \, \Phi_0 \\ D \, \Delta \, \Phi_1 \end{pmatrix} 
+\dfrac{1}{2} \, \begin{pmatrix}
\Delta \, \Phi_{-1} \\ \Delta \, \Phi_0 \\ \Delta \, \Phi_1 \end{pmatrix}^T \, M \, \begin{pmatrix}
\Delta \, \Phi_{-1} \\ \Delta \, \Phi_0 \\ \Delta \, \Phi_1 \end{pmatrix} \, ,
\end{equation}
with $Q$ given in \eqref{eq:tele_fiveepoint} and $M$ given in \eqref{eq:symm_fiveepoint}.

\subsubsection{Second order extrapolation of type 1}

We can prove
\begin{proposition}
\label{prop5pts-1}
Let $r=1$, $h_0=1/2$, and consider the set of second order extrapolation conditions $\Delta \, \Phi_0 = D_0 \, \Delta \, \Phi_0 =0$. In (\ref{eq:E5pts_0}), we then have,
\begin{equation}\label{eq:E5pts_1}
E \, = \, \begin{pmatrix}
\Phi_0 \\ D \, \Phi_0 \\ D^2 \, \Phi_0 \end{pmatrix}^T \, B \, \begin{pmatrix}
\Phi_0 \\ D \, \Phi_0 \\ D^2 \, \Phi_0 \end{pmatrix} \, , \, B \, := \, \begin{pmatrix}
-z & \dfrac{z^2}{2}& \dfrac{\sigma}{2} \\
\dfrac{z^2}{2} & -\dfrac{z^3}{2}-\sigma &  \dfrac{-z\sigma}{2} \\
\dfrac{\sigma}{2} &  \dfrac{-z\sigma}{2} & \dfrac{2 \, \tau}{3} +\sigma \, \tau +\dfrac{5 \, z \, \sigma}{12} -\dfrac{z^2 \, \tau}{2} +2 \, \tau^2 -\dfrac{z^2 \, (1-z^2)}{12} 
\end{pmatrix} \, .
\end{equation}
In particular, the scheme is semi-bounded (i.e. $B$ is negative semi-definite) for all $z \in (0,1)$ at least in the following two special cases,
\begin{itemize}
 \item $\sigma=0$ and $\tau=-(1-z^2)/12$,
 \item $\sigma=z \, (1-z^2)/6$ and $\tau=-\sigma/2$.
\end{itemize}
\end{proposition}

\begin{proof}
The two extrapolation conditions readily yield,
\begin{equation*}
 D \, \Phi_1 \, =  \, D \, \Phi_0 \quad \quad \Delta \, \Phi_1\, =  \, \Delta \, \Phi_{-1} \, .
\end{equation*}
By combining the two we also get $D \, \Delta \, \Phi_1 = \Delta \, \Phi_{-1}$, and by definition we also have $\Phi_{-1} = \Phi_0 -D \, \Phi_0$ 
as well as $\Delta \, \Phi_{-1} =D^2 \, \Phi_0$. With these formulas, we can now simplify $E$ in (\ref{eq:E5pts_0}) into,
\begin{equation*}
E \, = \, \dfrac{1}{2} \, \begin{pmatrix}
\Phi_0 -D \, \Phi_0 \\ D \, \Phi_0 \\ D^2 \, \Phi_0 \\ -D^2 \, \Phi_0 \end{pmatrix}^T \, Q \, \begin{pmatrix}
\Phi_0 -D \, \Phi_0 \\ D \, \Phi_0 \\ D^2 \, \Phi_0 \\ -D^2 \, \Phi_0 \end{pmatrix}
+\dfrac{1}{2} \, \begin{pmatrix}
\Phi_0 \\ D \, \Phi_0 \\ 0 \\ D^2 \, \Phi_0 \end{pmatrix}^T \, Q \, \begin{pmatrix}
\Phi_0 \\ D \, \Phi_0 \\ 0 \\ D^2 \, \Phi_0 \end{pmatrix}
+\dfrac{1}{2} \, \begin{pmatrix}
D^2 \, \Phi_0 \\ 0 \\ D^2 \, \Phi_0 \end{pmatrix}^T \, M \, \begin{pmatrix}
D^2 \, \Phi_0 \\ 0 \\ D^2 \, \Phi_0 \end{pmatrix} \, ,
\end{equation*}
which in turn, after some straightforward algebra, leads to (\ref{eq:E5pts_1}).

For the semi-boundedness part of the Proposition, let us first consider the case $\sigma=0$ and $\tau=-(1-z^2)/12$. The general expression for the matrix $B$ 
in (\ref{eq:E5pts_1}) reduces to,
$$
B \, := \, \begin{pmatrix}
-z & z^2/2 & 0 \\
z^2/2 & -\dfrac{z^3}{2} & 0 \\
0 & 0 & -\dfrac{(1-z^2) \, (3+4\, z^2)}{72} \end{pmatrix} \, .
$$
It is a simple exercise to verify that the upper left $2 \times 2$ block is negative definite for $z \in (0,1)$, hence $B$ is negative definite for $z \in (0,1)$. We 
get in that case the conclusion of Proposition \ref{prop5pts-1} as an immediate consequence. We now focus on the case $\sigma=z \, (1-z^2)/6$ and $\tau 
=-\sigma/2$, for which the matrix $B$ reads:
\begin{equation*}
B \, = \, -z \,  \begin{pmatrix}
1 & -z/2 & -(1-z^2)/12 \\
-z/2 & \dfrac{1+2\, z^2}{6} & z \, (1-z^2)/12 \\
-(1-z^2)/12 & z \, (1-z^2)/12 & \dfrac{(1-z^2) \, (1+z) \, (4 -3\, z)}{72} \end{pmatrix} \, .
\end{equation*}
We then compute:
\begin{multline*}
\begin{pmatrix}
U \\ V \\ W \end{pmatrix} ^T B \, \begin{pmatrix}
U \\ V \\ W \end{pmatrix} \\
= \, -z \left( \left( U -\dfrac{z}{2} \, V -\dfrac{1-z^2}{12} \, W \right)^2 +\dfrac{z^2}{4} \, V^2 
+\dfrac{(1-z^2)}{6} \, \left( \left( V+\dfrac{z}{4} \, W \right)^2 +\dfrac{14 +4\, z-8\, z^2}{48} \, W^2 \right) \right) \, , 
\end{multline*}
thereby showing that $B$ is negative semi-definite for $z \in (0,1)$. This completes the proof of Proposition \ref{prop5pts-1}.
\end{proof}

\subsubsection{Outflow boundary: second order extrapolation of type 2}

We examine in this paragraph a second set of outflow boundary conditions with second order accuracy.

\begin{proposition}
\label{prop5pts-2}
Let $r=1$, $h_0=1/2$, and consider the set of second order extrapolation conditions $\Delta \, \Phi_0 = \Delta^2 \, \Phi_0 =0$. Then we have
\begin{equation}\label{eq:E5pts_2}
E \, = \, \begin{pmatrix}
\Phi_0 \\ D \, \Phi_0 \\ D^2 \, \Phi_0 \end{pmatrix}^T \, B \, \begin{pmatrix}
\Phi_0 \\ D \, \Phi_0 \\ D^2 \, \Phi_0 \end{pmatrix} \, ,\quad B \, := \, \begin{pmatrix}
-z  & z^2/2 & -\tau \\
z^2/2 & -\frac{z^3}{2}-\sigma & z \, \tau \\
-\tau & z \, \tau & \frac{2 \, \tau}{3} +\sigma \, \tau+\frac{5 \, z \, \sigma}{12} -\frac{z^2 \, \tau}{2} +\frac{\sigma^2}{2} -\frac{z^2 \, (1-z^2)}{12}\end{pmatrix} .
\end{equation}
In particular, the scheme is semi-bounded (i.e. $B$ is negative semi-definite) for all $z \in (0,1)$ at least in the following two special cases,
\begin{itemize}
 \item $\sigma=z \, (1-z^2)/6$ and $\tau=-\sigma/2$,
 \item $\sigma=z \, (1-z^2)/6$ and $\tau=-z^2 \, (1-z^2)/24$.
\end{itemize}
\end{proposition}

\begin{proof}
We use \eqref{eq:E5pts_0} and insert the second order extrapolation conditions $\Delta \, \Phi_0 = \Delta^2 \, \Phi_0 =0$, from which we deduce 
the relations $D\, \Phi_1 =D\, \Phi_0$, and $D \, \Delta \, \Phi_1 =\Delta \, \Phi_1 =-\Delta \, \Phi_{-1}$. The boundary contribution $E$ in the energy 
balance \eqref{defE} reads
\begin{multline*}
E \, = \, \dfrac{1}{2} \, \begin{pmatrix}
\Phi_0 -D \, \Phi_0 \\ D \, \Phi_0 \\ D^2 \, \Phi_0 \\ -D^2 \, \Phi_0 \end{pmatrix}^T \, Q \, \begin{pmatrix}
\Phi_0 -D \, \Phi_0 \\ D \, \Phi_0 \\ D^2 \, \Phi_0 \\ -D^2 \, \Phi_0 \end{pmatrix} 
+\dfrac{1}{2} \, \begin{pmatrix}
\Phi_0 \\ D \, \Phi_0 \\ 0 \\ -D^2 \, \Phi_0 \end{pmatrix}^T \, Q \, \begin{pmatrix}
\Phi_0 \\ D \, \Phi_0 \\ 0 \\ -D^2 \, \Phi_0 \end{pmatrix} \\
+\dfrac{1}{2} \, \begin{pmatrix}
D^2 \, \Phi_0 \\ 0 \\ -D^2 \, \Phi_0 \end{pmatrix}^T \, M \, \begin{pmatrix}
D^2 \, \Phi_0 \\ 0 \\ -D^2  \, \Phi_0 \end{pmatrix} \, ,
\end{multline*}
which leads to (\ref{eq:E5pts_2}).

We first focus on the case $\sigma=z \, (1-z^2)/6$ and $\tau =-\sigma/2$, for which the matrix $B$ reads:
\begin{equation*}
B \, = \, -z\begin{pmatrix}
1 & -z/2 & -(1-z^2)/12 \\
-z/2 & \dfrac{1+2\, z^2}{6} & z \, (1-z^2)/12 \\
-(1-z^2)/12 & z \, (1-z^2)/12 & \dfrac{(1-z^2) \, (1+z) \, (4 -3\, z)}{72} \end{pmatrix} \, ,
\end{equation*}
which we have already analyzed in the proof of Proposition \ref{prop5pts-1}. In particular, we have already shown that $B$ is negative semi-definite for 
any $z \in (0,1)$ and the result of Proposition \ref{prop5pts-2} follows in that case. Finally, and most importantly, the fourth order Strang scheme $\sigma 
=z \, (1-z^2)/6$, $\tau =-z^2 \, (1-z^2)/24$ yields the expression,
\begin{equation*}
B \, = \, -z \, \begin{pmatrix}
1  & -z/2 & -z \, (1-z^2)/24 \\
-z/2 & \dfrac{1+2 \, z^2}{6} & z^2 \, (1-z^2)/24 \\
-z \, (1-z^2)/24 & z^2 \, (1-z^2)/24 & \frac{z \, (1-z^2) \, (4-z^2+z\, (1-z^2))}{144} \end{pmatrix} \, .
\end{equation*}
We thus compute:
\begin{multline*}
\begin{pmatrix}
U \\ V \\ W \end{pmatrix} ^T B \, \begin{pmatrix}
U \\ V \\ W \end{pmatrix} \\
= \, -z \left( \left( U -\dfrac{z}{2} \, V -\dfrac{z \, (1-z^2)}{24} \, W \right)^2 +\dfrac{1}{6} \, V^2 
+\dfrac{z^2}{12} \, \left( V +\dfrac{1-z^2}{4} \, W \right)^2 +\dfrac{z \, (1-z^2) \, (4-z^2)}{144} \, W^2 \right) \, ,
\end{multline*}
showing that $B$ is negative semi-definite for any $z \in (0,1)$. Hence the result of Proposition \ref{prop5pts-2} follows in the case of the fourth order 
Strang scheme.
\end{proof}

In view of all above results, the energy method seems to be a rather efficient tool to prove stability estimates for numerical schemes that are contractive 
in the whole space $\mathbb{Z}$ combined with some carefully selected extrapolation conditions at an outflow boundary. It is the purpose of the next section 
to illustrate the limitations of this energy approach, namely that it can not be applied to the more often considered translation invariant extrapolation conditions. 
For such boundary conditions and more general schemes (based for instance on multistep quadrature methods in time), the general theory for proving stability 
of numerical boundary conditions initiated in \cite{gks} remains the only available one.

\section{On the limitation of the energy method for analyzing numerical boundary conditions}

In this paragraph, we consider the fourth order Strang scheme \cite{strang}:
\begin{equation}\label{schemaS4}
A(z) \, := \, I -z \, D_0 +\dfrac{z^2}{2} \, \Delta +\dfrac{z \, (1-z^2)}{6} \, D_0 \, \Delta -\dfrac{z^2 \, (1-z^2)}{24} \, \Delta^2 \, .
\end{equation}
With a prescribed five point stencil, the finite difference scheme \eqref{schemaS4} is the only one that achieves fourth order accuracy (with respect 
to both space and time), just like the Lax-Wendroff scheme is the only three point scheme that achieves second order accuracy. Our goal is to study 
the semi-boundedness of \eqref{schemaS4} when implemented on a half line $\mathbb{Z}^-$ with extrapolation numerical boundary conditions, be 
they for instance of order $1$ or $2$. The extrapolation boundary conditions considered in \cite{kreissproc,goldberg,CL} are, in the terminology of 
\cite{goldberg-tadmor1,goldberg-tadmor}, \emph{translatory}, meaning that they are of the exact same form in each ghost cell (as opposed to the 
extrapolation conditions considered in Propositions \ref{prop5pts-1} and \ref{prop5pts-2}). The main results in \cite{kreissproc,kreiss1,goldberg,CL} 
show that, whatever the extrapolation order at the outflow boundary, the corresponding operator on $\ell^2(-\infty,0)$ is power-bounded. The proof 
of this result in \cite{kreissproc,kreiss1,goldberg} relies on the normal mode decomposition and power-boundedness follows from the general result 
in \cite{wu}, while the more direct proof of the same result in \cite{CL} relies first on the energy method for the Dirichlet boundary condition and on 
an induction argument with respect to the extrapolation order at the outflow boundary. Our goal below is to determine whether, as in Corollary 
\ref{cor2}, such stability estimates can be achieved by means of a straightforward energy method without using ``auxiliary problems''. The answer 
is \emph{negative}, see Theorems \ref{thm1} and \ref{thm2} below, which seems to indicate that the induction argument in \cite{CL} is more or less 
the shortest way to derive stability estimates for translatory extrapolation numerical boundary conditions. Examples of uniformly stable, though non 
dissipative, boundary conditions have been known for quite some time in the context of hyperbolic partial differential equations, see for instance the 
examples provided in \cite{BRSZ,Benoit}. As far as we are aware of, the example given by Theorem \ref{thm1} seems to be the first in the fully discrete 
setting.

Our first main result for the Strang scheme \eqref{schemaS4} is the following.

\begin{theorem}
\label{thm1}
There does not exist a parameter $z_0>0$, an integer $r \in \N$, and a continuous map $H$ from $[0,z_0]$ with values in $\mathcal{M}_r(\R)$ such 
that for all $z \in [0,z_0]$, $H(z)$ is a symmetric matrix satisfying the following property: for any sequence $\boldsymbol{\Phi} \in \ell^2(-\infty,2)$ satisfying 
the first order extrapolation boundary condition $\Phi_2=\Phi_1=\Phi_0$, the following energy inequality holds:
\begin{equation}
\label{toto}
\sum_{j \le -r} \Psi_j^2 +\begin{pmatrix}
\Psi_{-r+1} \\
\vdots \\
\Psi_0 \end{pmatrix}^T \, H(z) \, \begin{pmatrix}
\Psi_{-r+1} \\
\vdots \\
\Psi_0 \end{pmatrix} -\sum_{j \le -r} \Phi_j^2 -\begin{pmatrix}
\Phi_{-r+1} \\
\vdots \\
\Phi_0 \end{pmatrix}^T \, H(z) \, \begin{pmatrix}
\Phi_{-r+1} \\
\vdots \\
\Phi_0 \end{pmatrix} \, \le \, 0 \, ,
\end{equation}
with the sequence $\boldsymbol{\Psi} \in \ell^2(\mathbb{Z}^-)$ being defined by:
$$
\forall \, j \le 0 \, ,\quad \Psi_j \, := \, A(z) \, \Phi_j \, = \, \Phi_j -z \, D_0 \, \Phi_j +\dfrac{z^2}{2} \, \Delta \, \Phi_j 
+\dfrac{z \, (1-z^2)}{6} \, D_0 \, \Delta \, \Phi_j -\dfrac{z^2 \, (1-z^2)}{24} \, \Delta^2 \, \Phi_j \, .
$$
\end{theorem}

\noindent Let us observe that we do not even assume the matrix $H(z)$ to be positive definite, which would be necessary to make the quantity:
$$
\sum_{j \le -r} \Phi_j^2 +\begin{pmatrix}
\Phi_{-r+1} \\
\vdots \\
\Phi_0 \end{pmatrix}^T \, H(z) \, \begin{pmatrix}
\Phi_{-r+1} \\
\vdots \\
\Phi_0 \end{pmatrix} \, ,
$$
the square of a norm on the space of sequences $\boldsymbol{\Phi} \in \ell^2(-\infty,2)$ with the prescribed boundary conditions. The obstacle for 
the existence of $H(z)$ in Theorem \ref{thm1} comes from the small values of $z$. It could very well be that for some $z \in (0,1)$, one can construct 
a real symmetric positive definite matrix $H(z)$ of size $r \in \N$ such that one has the optimal energy balance:
$$
\sum_{j \le -r} \Psi_j^2 +\begin{pmatrix}
\Psi_{-r+1} \\
\vdots \\
\Psi_0 \end{pmatrix}^T \, H(z) \, \begin{pmatrix}
\Psi_{-r+1} \\
\vdots \\
\Psi_0 \end{pmatrix} -\sum_{j \le -r} \Phi_j^2 -\begin{pmatrix}
\Phi_{-r+1} \\
\vdots \\
\Phi_0 \end{pmatrix}^T \, H(z) \, \begin{pmatrix}
\Phi_{-r+1} \\
\vdots \\
\Phi_0 \end{pmatrix}
+c \, (\Phi_{-1}^2 +\Phi_0^2) \, \le \, 0 \, ,
$$
with $c$ a positive constant. Theorem \ref{thm1} shows, however, that such a construction will not be possible for all values of $z \in (0,1)$ with 
the additional requirement that $H(z)$ extends continuously to $z=0$ (as was the case, for instance, for the Lax-Wendroff scheme with second 
order extrapolation by Corollary \ref{cor2}).

\begin{proof}[Proof of Theorem \ref{thm1}]
The proof of Theorem \ref{thm1} is based on an induction argument with respect to the integer $r$. As will follow from the argument below, it appears that 
the initial step of the induction argument corresponds to $r=2$. The first two cases $r=0$ and $r=1$ are dealt with separately, though the argument is similar 
to the one for $r=2$. Let us start with a general argument from which the conclusion of Theorem \ref{thm1} will follow.
\bigskip

$\bullet$ \underline{The derivative of the energy balance with respect to $z$ at $0$.} In all what follows, we assume that the opposite statement to 
Theorem \ref{thm1} holds, meaning that we assume that there exist a parameter $z_0>0$, an integer $r \in \N$, and a continuous map $H$ from 
$[0,z_0]$ with values in $\mathcal{M}_r(\R)$ such that for any $z \in [0,z_0]$, $H(z)$ is a real symmetric matrix with the previously stated property. 
Let us consider a sequence $\boldsymbol{\Phi} \in \ell^2(-\infty,2)$ that satisfies $\Phi_2=\Phi_1=\Phi_0$, and let us then denote with $f(z)$ the 
quantity on the left hand side of \eqref{toto}. Since the sequence $\boldsymbol{\Psi}(z)$ depends in a $\mathcal{C}^1$ fashion on $z$ in $\ell^2$, 
and since $H$ is continuous at $0$, the function $f$ is differentiable at $0$ and it holds that,
\begin{equation}
\label{derivee1}
f'(0) \, = \, \sum_{j \le -r} 2 \, \Psi_j(0) \, \Psi_j'(0) +2 \, \begin{pmatrix}
\Psi_{-r+1}'(0) \\
\vdots \\
\Psi_0'(0) \end{pmatrix}^T \, H(0) \, \begin{pmatrix}
\Phi_{-r+1} \\
\vdots \\
\Phi_0 \end{pmatrix} \, .
\end{equation}
Substituting the value of $\Psi_j'(0)$, \eqref{derivee1} can be expressed in terms of the sequence $\boldsymbol{\Phi}$ only, i.e.
\begin{align*}
f'(0) \, =& \, \sum_{j \le -r} -2 \, \Phi_j \, D_0 \, \Phi_j +\dfrac{1}{3} \, \Phi_j \, D_0 \, \Delta \, \Phi_j \\
& \, +\begin{pmatrix}
-(\Phi_{-r+2}-\Phi_{-r}) +(\Phi_{-r+3}-2\, \Phi_{-r+2}+2 \, \Phi_{-r}-\Phi_{-r-1})/6 \\
\vdots \\
-(\Phi_1-\Phi_{-1}) +(\Phi_2-2\, \Phi_1+2 \, \Phi_{-1}-\Phi_{-2})/6 \end{pmatrix}^T \, H(0) \, \begin{pmatrix}
\Phi_{-r+1} \\
\vdots \\
\Phi_0 \end{pmatrix} \\
=& \, -\dfrac{4}{3} \, \Phi_{-r} \, \Phi_{-r+1} +\dfrac{1}{6} \, (\Phi_{-r-1} \, \Phi_{-r+1} +\Phi_{-r} \, \Phi_{-r+2}) \\
& \, +\begin{pmatrix}
-(\Phi_{-r+2}-\Phi_{-r}) +(\Phi_{-r+3}-2\, \Phi_{-r+2}+2 \, \Phi_{-r}-\Phi_{-r-1})/6 \\
\vdots \\
-(\Phi_1-\Phi_{-1}) +(\Phi_2-2\, \Phi_1+2 \, \Phi_{-1}-\Phi_{-2})/6 \end{pmatrix}^T \, H(0) \, \begin{pmatrix}
\Phi_{-r+1} \\
\vdots \\
\Phi_0 \end{pmatrix} \, .
\end{align*}

We observe that $f$ vanishes at $z=0$ because $\Psi_j(0)=\Phi_j$ for all $j$, and $f$ takes nonpositive values on $[0,z_0]$ by \eqref{toto}. 
Hence $0$ is a maximum of $f$ on $[0,z_0]$ and the derivative $f'(0)$ is nonpositive. In other words, we have just obtained that there exists 
a real symmetric matrix $\mathcal{H}$ of size $r$ such that for any sequence $\boldsymbol{\Phi} \in \ell^2(-\infty,2)$ satisfying $\Phi_2=\Phi_1
=\Phi_0$, the following inequality holds:
\begin{multline}
\label{derivee}
-\dfrac{4}{3} \, \Phi_{-r} \, \Phi_{-r+1} +\dfrac{1}{6} \, (\Phi_{-r-1} \, \Phi_{-r+1} +\Phi_{-r} \, \Phi_{-r+2}) \\
+\begin{pmatrix}
-(\Phi_{-r+2}-\Phi_{-r}) +(\Phi_{-r+3}-2\, \Phi_{-r+2}+2 \, \Phi_{-r}-\Phi_{-r-1})/6 \\
\vdots \\
-(\Phi_1-\Phi_{-1}) +(\Phi_2-2\, \Phi_1+2 \, \Phi_{-1}-\Phi_{-2})/6 \end{pmatrix}^T \, \mathcal{H} \, \begin{pmatrix}
\Phi_{-r+1} \\
\vdots \\
\Phi_0 \end{pmatrix} \, \le \, 0 \, .
\end{multline}
It remains to examine the consequences of \eqref{derivee}, which is where the specific value of $r$ comes into play because the values 
$\Phi_1$ and $\Phi_2$ are not arbitrary due to the extrapolation boundary conditions.
\bigskip

$\bullet$ \underline{The case $r=0$.} In that case, the inequality \eqref{derivee} reduces to:
$$
-\dfrac{4}{3} \, \Phi_0 \, \Phi_1 +\dfrac{1}{6} \, (\Phi_{-1} \, \Phi_1 +\Phi_0 \, \Phi_2) \, \le \, 0 \, .
$$
Using the extrapolation boundary conditions $\Phi_2=\Phi_1=\Phi_0$, we end up with:
$$
-\dfrac{7}{6} \, \Phi_0^2 \, +\dfrac{1}{6} \, \Phi_0 \, \Phi_{-1} \, \le \, 0 \, ,
$$
which is obviously impossible since the values $\Phi_0,\Phi_{-1}$ are arbitrary. This means that we cannot use the standard $\ell^2$ norm on $\mathbb{Z}^-$ 
for proving the stability of \eqref{schemaS4} with first order extrapolation condition. Let us now deal with the next case ($r=1$) in the induction argument.
\bigskip

$\bullet$ \underline{The case $r=1$.} In that case, the inequality \eqref{derivee} reduces to:
\begin{equation*}
-\dfrac{4}{3} \, \Phi_{-1} \, \Phi_0 +\dfrac{1}{6} \, (\Phi_{-2} \, \Phi_0+\Phi_{-1} \, \Phi_1) 
-\mathcal{H} \, \Phi_0 \, (\Phi_1-\Phi_{-1}) +\dfrac{\mathcal{H}}{6} \, \Phi_0 \, (\Phi_2-2\, \Phi_1+2 \, \Phi_{-1}-\Phi_{-2}) \, \le \, 0 \, ,
\end{equation*}
where $\mathcal{H}$ is a real number. After using the boundary conditions $\Phi_2=\Phi_1=\Phi_0$, we get:
\begin{equation}
\label{ineg1-1}
-\dfrac{7}{6} \, \Phi_{-1} \, \Phi_0 +\dfrac{1}{6} \, \Phi_{-2} \, \Phi_0 -\mathcal{H} \, \Phi_0 \, (\Phi_0-\Phi_{-1}) 
-\dfrac{\mathcal{H}}{6} \, \Phi_0 \, (\Phi_0-2 \, \Phi_{-1}+\Phi_{-2}) \, \le \, 0 \, ,
\end{equation}
where now the three values $\Phi_0,\Phi_{-1},\Phi_{-2}$ are arbitrary. It is useful to introduce the new variables:
$$
y_3 \, := \, \Phi_0 \, ,\quad y_2 \, := \, \Phi_0-\Phi_{-1} \, ,\quad y_1 \, := \, \Phi_0-2\, \Phi_{-1}+\Phi_{-2} \, ,
$$
with which \eqref{ineg1-1} is rewritten as follows:
\begin{equation}
\label{ineg1-2}
\forall \, y \in \R^3 \, , \quad -y_3^2 +\left( \dfrac{5}{6} -\mathcal{H} \right) \, y_3 \, y_2 +\dfrac{1-\mathcal{H}}{6} \, y_3 \, y_1 \, \le \, 0 \, ,
\end{equation}
The latter inequality is obviously impossible since the only available parameter $\mathcal{H}$ should equal both $1$ and $5/6$ to cancel 
the off-diagonal terms $y_3 \, y_2$ and $y_3 \, y_1$. This completes the proof of Theorem \ref{thm1} in the case $r=1$. Let us now deal 
with the case $r=2$.
\bigskip

$\bullet$ \underline{The case $r=2$.} This is really the starting point of the induction argument, and we shall borrow the methodology introduced 
for the case $r=1$. Namely, when $r$ equals $2$, the inequality \eqref{derivee} reduces to:
\begin{equation}
\label{ineg2-1'}
-\dfrac{4}{3} \, \Phi_{-2} \, \Phi_{-1} +\dfrac{1}{6} \, (\Phi_{-3} \, \Phi_{-1}+\Phi_{-2} \, \Phi_0) 
+\begin{pmatrix}
-(\Phi_0-\Phi_{-2}) +(\Phi_1-2\, \Phi_0+2 \, \Phi_{-2}-\Phi_{-3})/6 \\
-(\Phi_1-\Phi_{-1}) +(\Phi_2-2\, \Phi_1+2 \, \Phi_{-1}-\Phi_{-2})/6 \end{pmatrix}^T \, \mathcal{H} \, \begin{pmatrix}
\Phi_{-1} \\
\Phi_0 \end{pmatrix} \, \le \, 0 \, .
\end{equation}
After using the boundary conditions $\Phi_2=\Phi_1=\Phi_0$, \eqref{ineg2-1'} reduces to:
\begin{equation}
\label{ineg2-1}
-\dfrac{4}{3} \, \Phi_{-2} \, \Phi_{-1} +\dfrac{1}{6} \, (\Phi_{-3} \, \Phi_{-1}+\Phi_{-2} \, \Phi_0) 
-\begin{pmatrix}
(\Phi_0-\Phi_{-2}) +(\Phi_0-2 \, \Phi_{-2}+\Phi_{-3})/6 \\
(\Phi_0-\Phi_{-1}) +(\Phi_0-2 \, \Phi_{-1}+\Phi_{-2})/6 \end{pmatrix}^T \, \mathcal{H} \, \begin{pmatrix}
\Phi_{-1} \\
\Phi_0 \end{pmatrix} \, \le \, 0 \, .
\end{equation}

Let us extend the strategy used in the case $r=1$, and introduce the new variables:
$$
y_4 \, := \, \Phi_0 \, ,\quad y_3 \, := \, D \, \Phi_0 \, ,\quad y_2 \, := \, D^2 \, \Phi_0 \, ,\quad y_1 \, := \, D^3 \, \Phi_0 \, .
$$
The inequality \eqref{ineg2-1} can be equivalently rewritten as:
\begin{equation}
\label{ineg2-2}
-y_4^2 +3 \, y_4\, y_3 -\dfrac{2}{3} \, y_4 \, y_2 -\dfrac{1}{6} \, y_4 \, y_1 -\dfrac{13}{6} \, y_3^2 +\dfrac{5}{6} \, y_3 \, y_2 +\dfrac{1}{6} \, y_3 \, y_1 
-\dfrac{1}{6} \, \begin{pmatrix}
13 \, y_3 -5 \, y_2 -y_1 \\
6 \, y_3 +y_2 \end{pmatrix}^T \, \mathcal{H} \, \begin{pmatrix}
y_4 -y_3 \\
y_4 \end{pmatrix} \, \le \, 0 \, ,
\end{equation}
where \eqref{ineg2-2} holds for all $y \in \R^4$, since $\Phi_0,\Phi_{-1},\Phi_{-2},\Phi_{-3}$ in \eqref{ineg2-1} are arbitrary. It is useful at this stage to 
introduce the coefficients of the symmetric matrix $\mathcal{H}$, and we thus write:
$$
\mathcal{H} \, = \, \begin{pmatrix}
h_{11} & h_{12} \\
h_{12} & h_{22} \end{pmatrix} \, .
$$

The quadratic form in $y \in \R^4$ on the left hand side of \eqref{ineg2-2} is nonpositive, and furthermore it has no $y_1^2$ term. This implies that the 
coefficients of the cross products $y_4 \, y_1$ and $y_3 \, y_1$ must vanish. Computing those coefficients, we get:
$$
h_{11} \, = \, 1 \, ,\quad h_{12} \, = \, 0 \, ,
$$
which means that the matrix $\mathcal{H}$ reads:
$$
\mathcal{H} \, = \, \begin{pmatrix}
1 & 0 \\
0 & h_{22} \end{pmatrix} \, ,
$$
and then \eqref{ineg2-2} reduces to:
$$
-y_4^2 +\left( \dfrac{5}{6} -h_{22} \right) \, y_4 \, y_3 +\dfrac{1-h_{22}}{6} \, y_4 \, y_2 \, \le \, 0 \, ,
$$
which is nothing else but the inequality \eqref{ineg1-2} we had obtained in the analysis of the case $r=1$ except for the shift in the indeces 
(one should only substitute $(y_4,y_3,y_2)$ in place of $(y_3,y_2,y_1)$ in \eqref{ineg1-2}, and $h_{22}$ in place of $\mathcal{H}$). As already 
observed in the analysis of the case $r=1$, we are led to a contradiction, which completes the proof of Theorem \ref{thm1} in the case $r=2$.
\bigskip

$\bullet$ \underline{The general case $r \ge 3$.} We go back to \eqref{derivee} and assume $r \ge 3$. In particular, the first line on the left 
hand side of \eqref{derivee} does not involve the ghost cell values $\Phi_1,\Phi_2$. Substituting the first order extrapolation boundary conditions 
$\Phi_2=\Phi_1=\Phi_0$ in \eqref{derivee}, we get the inequality:
\begin{multline}
\label{derivee3}
-\dfrac{4}{3} \, \Phi_{-r} \, \Phi_{-r+1} +\dfrac{1}{6} \, (\Phi_{-r-1} \, \Phi_{-r+1}+\Phi_{-r} \, \Phi_{-r+2}) \\
+\begin{pmatrix}
-(\Phi_{-r+2}-\Phi_{-r}) +(\Phi_{-r+3}-2\, \Phi_{-r+2}+2 \, \Phi_{-r}-\Phi_{-r-1})/6 \\
\vdots \\
-(\Phi_{-1}-\Phi_{-3}) +(\Phi_0-2\, \Phi_{-1}+2 \, \Phi_{-3}-\Phi_{-4})/6 \\
-(\Phi_0-\Phi_{-2}) -(\Phi_0-2 \, \Phi_{-2}+\Phi_{-3})/6 \\
-(\Phi_0-\Phi_{-1}) -(\Phi_0-2 \, \Phi_{-1}+\Phi_{-2})/6 \end{pmatrix}^T \, \mathcal{H} \, \begin{pmatrix}
\Phi_{-r+1} \\
\vdots \\
\Phi_{-2} \\
\Phi_{-1} \\
\Phi_0 \end{pmatrix} \, \le \, 0 \, .
\end{multline}
We introduce the variable $y \in \R^{r+2}$ defined by:
$$
\forall \, \ell \, = \, 1,\dots,r+2 \, ,\quad y_\ell \, := \, D^{r+2-\ell} \, \Phi_0 \, ,
$$
which, conversely, corresponds to:
$$
\forall \, \ell \, = \, 0,\dots,r+1 \, ,\quad \Phi_{-\ell} \, = \, D^\ell \, y_{r+2} \, .
$$
Using from now on the coordinates of $y \in \R^{r+2}$ as free parameters, \eqref{derivee3} reads:
\begin{multline}
\label{derivee4}
-\dfrac{4}{3} \, D^{r} \, y_{r+2} \, D^{r-1} \, y_{r+2} +\dfrac{1}{6} \, 
(D^{r+1} \, y_{r+2} \, D^{r-1} \, y_{r+2} +D^{r} \, y_{r+2} \, D^{r-2} \, y_{r+2}) \\
+\begin{pmatrix}
-2\, D^{r-2} \, y_{r+1} +D^{r-2} \, y_{r} +(2 \, D^{r-3} \, y_{r-1} -D^{r-3} \, y_{r-2})/6 \\
\vdots \\
-2\, D \, y_{r+1} +D \, y_{r} +(2 \, y_{r-1} -y_{r-2})/6 \\
-(13 \, y_{r+1}-5 \, y_{r} -y_{r-1})/6 \\
-y_{r+1} -y_{r}/6 \end{pmatrix}^T \, \mathcal{H} \, \begin{pmatrix}
D^{r-1} \, y_{r+2} \\
\vdots \\
D^2 \, y_{r+2} \\
D \, y_{r+2} \\
y_{r+2} \end{pmatrix} \, \le \, 0 \, .
\end{multline}

We are not going to compute all the coefficients of the quadratic form (in $y$) arising on the left hand side of \eqref{derivee4}. It is useful however 
to observe that the first two coordinates $y_1$ and $y_2$ of $y$ do not appear in the expressions of $D^{r-1} \, y_{r+2}$, \dots, $D \, y_{r+2}$, $y_{r+2}$. 
Therefore, if we rewrite the quadratic form (in $y$) arising on the left hand side of \eqref{derivee4} as $y^T \, S \, y$, with $S$ a real symmetric matrix of 
size $r+2$, then not only $S$ is nonpositive because of \eqref{derivee4}, but $S$ also reads:
$$
S \, = \, \begin{pmatrix}
\widetilde{S} & \Upsilon_2 & \Upsilon_1 \\
\Upsilon_2^T & 0 & 0 \\
\Upsilon_1^T & 0 & 0 \end{pmatrix} \, ,
$$
with $\Upsilon_1,\Upsilon_1 \in \R^r$, and $\widetilde{S}$ a real symmetric matrix of size $r$. Since $S$ is nonnegative, we must necessarily have 
$\Upsilon_1=\Upsilon_2=0$. In other words, this means that no cross product of the form $y_1 \, y_3,\dots,y_1 \, y_{r+2}$ or $y_2 \, y_3,\dots,y_2 \, y_{r+2}$ 
arises on the left hand side of \eqref{derivee4}, or, equivalently, that the quantity on the left hand side of \eqref{derivee4} does not depend on $(y_1,y_2)$. 
Computing the partial derivative with respect to $y_1$, we get the relation:
$$
\dfrac{(-1)^{r+1}}{6} \, \left( (1-\mathcal{H}_{11}) \, D^{r-1} \, y_{r+2} +\sum_{\ell=2}^{r} \mathcal{H}_{1\ell} \, D^{r-\ell} \, y_{r+2} \right) \, = \, 0 \, ,
$$
from which we deduce that the first line of $\mathcal{H}$ should read:
$$
\begin{pmatrix}
1 & 0 & \cdots & 0 \end{pmatrix} \, .
$$
Since $\mathcal{H}$ is symmetric, \eqref{derivee4} reduces to:
\begin{align}
& -\dfrac{4}{3} \, D^{r} \, y_{r+2} \, D^{r-1} \, y_{r+2} +\dfrac{1}{6} \, (D^{r+1} \, y_{r+2} \, D^{r-1} \, y_{r+2}+D^{r} \, y_{r+2} \, D^{r-2} \, y_{r+2}) \notag \\
&+ \Big( -2\, D^{r-2} \, y_{r+1} +D^{r-2} \, y_{r} +\dfrac{1}{3} \, D^{r-3} \, y_{r-1} -\dfrac{1}{6} \, D^{r-3} \, y_{r-2} \Big) \, D^{r-1} \, y_{r+2} \label{derivee5} \\
&+\begin{pmatrix}
-2\, D^{r-3} \, y_{r+1} +D^{r-3} \, y_{r} +(2 \, D^{r-4} \, y_{r-1} -D^{r-4} \, y_{r-2})/6 \\
\vdots \\
-2\, D \, y_{r+1} +D \, y_{r} +(2 \, y_{r-1} -y_{r-2})/6 \\
-(13 \, y_{r+1}-5 \, y_{r} -y_{r-1})/6 \\
-y_{r+1} -y_{r}/6 \end{pmatrix}^T \, \mathcal{H}_\sharp \, \begin{pmatrix}
D^{r-2} \, y_{r+2} \\
\vdots \\
D^2 \, y_{r+2} \\
D \, y_{r+2} \\
y_{r+2} \end{pmatrix} \, \le \, 0 \, ,\notag
\end{align}
where the real symmetric matrix $\mathcal{H}_\sharp$ of size $r-1$ corresponds to the block decomposition of $\mathcal{H}$:
$$
\mathcal{H} \, = \, \begin{pmatrix}
1 & 0 \\
0 & \mathcal{H}_\sharp \end{pmatrix} \, .
$$
We can simplify the first two lines of \eqref{derivee5} by using the relation:
\begin{multline*}
-2\, D^{r-2} \, y_{r+1} +D^{r-2} \, y_{r} +\dfrac{1}{3} \, D^{r-3} \, y_{r-1} -\dfrac{1}{6} \, D^{r-3} \, y_{r-2} \\
= \, -\big( D^{r-2} \, y_{r+2} -D^{r} \, y_{r+2} \big) +\dfrac{1}{6} \, \big( D^{r-3} \, y_{r+2} -2 \, D^{r-2} \, y_{r+2} 
+2 \, D^{r} \, y_{r+2} -D^{r+1} \, y_{r+2} \big) \, ,
\end{multline*}
and \eqref{derivee5} can be rewritten as:
\begin{multline}
\label{derivee6}
-\dfrac{4}{3} \, D^{r-1} \, y_{r+2} \, D^{r-2} \, y_{r+2} +\dfrac{1}{6} \, (D^{r} \, y_{r+2} \, D^{r-2} \, y_{r+2} +D^{r-1} \, y_{r+2} \, D^{r-3} \, y_{r+2}) \\
+\begin{pmatrix}
-2\, D^{r-3} \, y_{r+1} +D^{r-3} \, y_{r} +(2 \, D^{r-4} \, y_{r-1} -D^{r-4} \, y_{r-2})/6 \\
\vdots \\
-2\, D \, y_{r+1} +D \, y_{r} +(2 \, y_{r-1} -y_{r-2})/6 \\
-(13 \, y_{r+1}-5 \, y_{r} -y_{r-1})/6 \\
-y_{r+1} -y_{r}/6 \end{pmatrix}^T \, \mathcal{H}_\sharp \, \begin{pmatrix}
D^{r-2} \, y_{r+2} \\
\vdots \\
D^2 \, y_{r+2} \\
D \, y_{r+2} \\
y_{r+2} \end{pmatrix} \, \le \, 0 \, .
\end{multline}
Shfiting the indeces in the variables, that is introducing the vector:
$$
(\tilde{y}_{r+1},\dots,\tilde{y}_1) \, := \, (y_{r+2},\dots,y_2) \, ,
$$
and forgetting about the tilde, we see that \eqref{derivee6} is exactly the same as \eqref{derivee4} with the integer $r-1$ in place of $r$. 
By a finite induction process, we can therefore show that the validity of \eqref{derivee4} for some real symmetric matrix $\mathcal{H}$ of size 
$r$ implies the validity of \eqref{ineg2-2} (which is exactly \eqref{derivee4} in the particular case $r=2$), and we have already seen that 
this leads to a contradiction. The proof of Theorem \ref{thm1} is now complete.
\end{proof}

\noindent The exact same argument of proof can be used to deal with the case of second order extrapolation at the boundary. We shall not 
reproduce the proof here and leave the (minor) modifications to the interested reader. We thus only state the final result, which is entirely 
similar to Theorem \ref{thm1} above except for the extrapolation conditions.

\begin{theorem}
\label{thm2}
There does not exist a parameter $z_0>0$, an integer $r \in \N$, and a continuous map $H$ from $[0,z_0]$ with values in $\mathcal{M}_r(\R)$ 
such that for any $z \in [0,z_0]$, $H(z)$ is a symmetric matrix satisfying the following property: for any sequence $\boldsymbol{\Phi} \in \ell^2(-\infty,2)$ 
verifying the second order extrapolation boundary conditions $\Delta \, \Phi_1=\Delta \, \Phi_0=0$, with the sequence $\boldsymbol{\Psi}$ being 
defined by:
$$
\forall \, j \le 0 \, ,\quad \Psi_j \, := \, A(z) \, \Phi_j \, = \, \Phi_j -z \, D_0 \, \Phi_j +\dfrac{z^2}{2} \, \Delta \, \Phi_j 
+\dfrac{z \, (1-z^2)}{6} \, D_0 \, \Delta \, \Phi_j -\dfrac{z^2 \, (1-z^2)}{24} \, \Delta^2 \, \Phi_j \, ,
$$
then the following energy inequality holds:
\begin{equation*}
\sum_{j \le -r} \Psi_j^2 +\begin{pmatrix}
\Psi_{-r+1} \\
\vdots \\
\Psi_0 \end{pmatrix}^T \, H(z) \, \begin{pmatrix}
\Psi_{-r+1} \\
\vdots \\
\Psi_0 \end{pmatrix} -\sum_{j \le -r} \Phi_j^2 -\begin{pmatrix}
\Phi_{-r+1} \\
\vdots \\
\Phi_0 \end{pmatrix}^T \, H(z) \, \begin{pmatrix}
\Phi_{-r+1} \\
\vdots \\
\Phi_0 \end{pmatrix} \, \le \, 0 \, .
\end{equation*}
\end{theorem}

\noindent Theorems \ref{thm1} and \ref{thm2} show that for the fourth order Strang scheme \eqref{schemaS4}, stability for the \emph{translatory} first 
or second order extrapolation conditions:
\begin{align*}
\text{\rm (first order)} \quad & \Phi_2 \, = \, \Phi_1 \, = \, \Phi_0 \, ,\\
\text{\rm (second order)} \quad & \Delta \, \Phi_1 \, = \, \Delta \, \Phi_0 \, = \, 0 \, ,
\end{align*}
cannot be obtained by a `straightforward' energy argument (at least for all relevant values of the CFL parameter $z$), meaning by the construction 
of an energy that is a finite rank perturbation of the identity that is non-increasing for the associated evolution operator. It is known nevertheless that 
these numerical boundary conditions satisfy the strong stability condition of \cite{gks}, see \cite{goldberg}. These two above examples clearly indicate 
that the theory initiated in \cite{gks} is the only one able to capture stability for numerical boundary conditions in general.

\appendix
\section{Higher order integration by parts decompositions}

This appendix is devoted to the proof of the relations \eqref{eq:tele_11}-\eqref{eq:tele_14} which we have used in the proof of Lemma \ref{lem2}. 
Let us start with the proof of formula \eqref{eq:tele_11}, which we rewrite here for the reader's convenience:
\begin{multline}
\label{appA-formule1}
2 \, D_0 \, \Phi_j \, D_0 \, \Delta \, \Phi_j \, = \, -2 \, (\Delta \, \Phi_j)^2 +\dfrac{1}{4} \, \left[ (D \, \Delta \, \Phi_j)^2 +(D \, \Delta \, \Phi_{j+1})^2 \right] \\
+\left[ (D \, \Phi_{j+1} -\dfrac{1}{4} \, D \, \Delta \, \Phi_{j+1}) \, (D \, \Delta \, \Phi_{j+1} +2 \, \Delta \, \Phi_j) \right] 
-\left[ (D \, \Phi_j -\dfrac{1}{4} \, D \, \Delta \, \Phi_j) \, (D \, \Delta \, \Phi_j +2 \, \Delta \, \Phi_{j-1}) \right] \, .
\end{multline}
We first apply the formula \eqref{eq:tele_04} to get:
\begin{multline}
\label{appA-formule1-1}
2 \, D_0 \, \Phi_j \, \Delta \, D_0 \, \Phi_j \, = \, -(D \, D_0 \, \Phi_j)^2 -(D \, D_0 \, \Phi_{j+1})^2 \\
+\left[ 2 \, D_0 \, \Phi_j \, D \, D_0 \, \Phi_{j+1} +(D \, D_0 \, \Phi_{j+1})^2 \right] -\left[ 2 \, D_0 \, \Phi_{j-1} \, D \, D_0 \, \Phi_j +(D \, D_0 \, \Phi_j)^2 \right] \, .
\end{multline}
and then rewrite the first line on the right hand side of \eqref{appA-formule1-1} (that is, the symmetric terms) as:
\begin{align*}
(D \, D_0 \, \Phi_j)^2 +(D \, D_0 \, \Phi_{j+1})^2 \, =& \, \dfrac{1}{4} \, (\Delta \, \Phi_{j-1} +\Delta \, \Phi_j)^2 +\dfrac{1}{4} \, (\Delta \, \Phi_j +\Delta \, \Phi_{j+1})^2 \\
=& \, \dfrac{1}{4} \, (-D \, \Delta \, \Phi_j +2 \, \Delta \, \Phi_j)^2 +\dfrac{1}{4} \, (2 \, \Delta \, \Phi_j +D \, \Delta \, \Phi_{j+1})^2 \\
=& \, 2 \, (\Delta \, \Phi_j)^2 +\Delta \, \Phi_j \, \Delta^2 \, \Phi_j +\dfrac{1}{4} \, \big( (D \, \Delta \, \Phi_j)^2 +(D \, \Delta \, \Phi_{j+1})^2 \big)
\end{align*}
We use again \eqref{eq:tele_04} for the term $\Delta \, \Phi_j \, \Delta^2 \, \Phi_j$ and get:
\begin{multline*}
(D \, D_0 \, \Phi_j)^2 +(D \, D_0 \, \Phi_{j+1})^2 \, = \, 2 \, (\Delta \, \Phi_j)^2 -\dfrac{1}{4} \, \big( (D \, \Delta \, \Phi_j)^2 +(D \, \Delta \, \Phi_{j+1})^2 \big) \\
+\left[ \dfrac{1}{2} (\Delta \, \Phi_{j+1})^2 -\dfrac{1}{2} (\Delta \, \Phi_j)^2 \right] -\left[ \dfrac{1}{2} (\Delta \, \Phi_j)^2 -\dfrac{1}{2} (\Delta \, \Phi_{j-1})^2 \right] \, .
\end{multline*}
Substituting this relation in the first line on the right hand side of \eqref{appA-formule1-1}, we have thus obtained the decomposition:
$$
2 \, D_0 \, \Phi_j \, \Delta \, D_0 \, \Phi_j \, = \, -2 \, (\Delta \, \Phi_j)^2 +\dfrac{1}{4} \, \big( (D \, \Delta \, \Phi_j)^2 +(D \, \Delta \, \Phi_{j+1})^2 \big) +T_j -T_{j-1} \, ,
$$
with
\begin{align*}
T_j \, :=& \, 2 \, D_0 \, \Phi_j \, D \, D_0 \, \Phi_{j+1} +(D \, D_0 \, \Phi_{j+1})^2 -\dfrac{1}{2} (\Delta \, \Phi_{j+1})^2 +\dfrac{1}{2} (\Delta \, \Phi_j)^2 \\
=& \, (D_0 \, \Phi_{j+1})^2 -(D_0 \, \Phi_j)^2 -\dfrac{1}{2} (\Delta \, \Phi_{j+1})^2 +\dfrac{1}{2} (\Delta \, \Phi_j)^2 \, .
\end{align*}
The decomposition \eqref{appA-formule1} then follows by just rewriting the quantity $T_j$ in terms of $D\, \Phi_{j+1}$, $\Delta \, \Phi_j$ and $D\, \Delta \, \Phi_{j+1}$. 
The (small) details are left to the reader.
\bigskip

We now wish to justify the purely telescopic formula \eqref{eq:tele_12}, which we also rewrite here for convenience:
\begin{multline}
\label{appA-formule2}
2 \, \Phi_j \, D_0 \, \Delta \, \Phi_j \, = \, \left[ \Phi_j \, (D \, \Delta \, \Phi_{j+1} +2 \, \Delta \, \Phi_j) +(\Delta \, \Phi_j -D \, \Phi_{j+1}) \, D \, \Phi_{j+1} \right] \\
- \left[ \Phi_{j-1} \, (D \, \Delta \, \Phi_j +2 \, \Delta \, \Phi_{j-1}) +(\Delta \, \Phi_{j-1} -D \, \Phi_j) \, D \, \Phi_j \right] \, .
\end{multline}
Let us write the left hand side of \eqref{appA-formule2} as:
$$
2 \, \Phi_j \, D_0 \, \Delta \, \Phi_j \, = \, \begin{pmatrix}
\Phi_{j+2} \\
\Phi_{j+1} \\
\Phi_j \\
\Phi_{j-1} \\
\Phi_{j-2} \end{pmatrix}^T \, \begin{pmatrix}
0 & 0 & 1/2 & 0 & 0 \\
0 & 0 & -1 & 0 & 0 \\
1/2 & -1 & 0 & 1 & -1/2 \\
0 & 0 & 1 & 0 & 0 \\
0 & 0 & -1/2 & 0 & 0 \end{pmatrix} \, \begin{pmatrix}
\Phi_{j+2} \\
\Phi_{j+1} \\
\Phi_j \\
\Phi_{j-1} \\
\Phi_{j-2} \end{pmatrix} \, ,
$$
and decompose the corresponding symmetric matrix in a telescopic way:
$$
\begin{pmatrix}
0 & 0 & 1/2 & 0 & 0 \\
0 & 0 & -1 & 0 & 0 \\
1/2 & -1 & 0 & 1 & -1/2 \\
0 & 0 & 1 & 0 & 0 \\
0 & 0 & -1/2 & 0 & 0 \end{pmatrix} \, = \, \begin{pmatrix}
0 & 0 & 1/2 & 0 & 0 \\
0 & 0 & -1 & 1/2 & 0 \\
1/2 & -1 & 0 & 0 & 0 \\
0 & 1/2 & 0 & 0 & 0 \\
0 & 0 & 0 & 0 & 0 \end{pmatrix} \, - \, \begin{pmatrix}
0 & 0 & 0 & 0 & 0 \\
0 & 0 & 0 & 1/2 & 0 \\
0 & 0 & 0 & -1 & 1/2 \\
0 & 1/2 & -1 & 0 & 0 \\
0 & 0 & 1/2 & 0 & 0 \end{pmatrix} \, .
$$
At this stage, we have obtained the telescopic decomposition:
$$
2 \, \Phi_j \, D_0 \, \Delta \, \Phi_j \, = \, \left[ \Phi_j \, \Phi_{j+2} -2 \, \Phi_j \, \Phi_{j+1} +\Phi_{j-1} \, \Phi_{j+1} \right] 
-\left[ \Phi_{j-1} \, \Phi_{j+1} -2 \, \Phi_{j-1} \, \Phi_j +\Phi_{j-2} \, \Phi_j \right] \, ,
$$
and the proof of \eqref{appA-formule2} follows by rewriting the telescopic term as:
\begin{align*}
\Phi_j \, \Phi_{j+2} -2 \, \Phi_j \, \Phi_{j+1} +\Phi_{j-1} \, \Phi_{j+1} \, =& \, \Phi_j \, \Delta \, \Phi_{j+1} +\Phi_j \, \Delta \, \Phi_j -D\, \Phi_j \, D \, \Phi_{j+1} \\
=& \, \Phi_j \, (D \, \Delta \, \Phi_{j+1} +2 \, \Delta \, \Phi_j) +(\Delta \, \Phi_j -D\, \Phi_{j+1}) \, D \, \Phi_{j+1} \, .
\end{align*}
This completes the proof of \eqref{appA-formule2}.
\bigskip

We now turn to the proof of the telescopic formula \eqref{eq:tele_13}, which we also rewrite here for convenience:
\begin{multline}
\label{appA-formule3}
2 \, D_0 \, \Phi_j \, \Delta^2 \, \Phi_j \, = \, \left[ -(D \, \Delta \, \Phi_{j+1} +\Delta \, \Phi_j) \, \Delta \, \Phi_j +2 \, D \, \Phi_{j+1} \, D \, \Delta \, \Phi_{j+1} \right] \\
- \left[ -(D \, \Delta \, \Phi_j +\Delta \, \Phi_{j-1}) \, \Delta \, \Phi_{j-1} +2 \, D \, \Phi_j \, D \, \Delta \, \Phi_j \right] \, .
\end{multline}
We expand the left hand side of \eqref{appA-formule3} as follows:
\begin{align*}
2 \, D_0 \, \Phi_j \, \Delta^2 \, \Phi_j \, =& \, (\Phi_{j+1}-\Phi_{j-1}) \, (\Phi_{j+2} -4 \, \Phi_{j+1} +6 \, \Phi_j -4 \, \Phi_{j-1} +\Phi_{j-2}) \\
=& \, 6 \, \Phi_j \, (\Phi_{j+1} -\Phi_{j-1}) -4 \, (\Phi_{j+1}^2 -\Phi_{j-1}^2) +(\Phi_{j+1}-\Phi_{j-1}) \, (\Phi_{j+2} +\Phi_{j-2}) \\
=& \, \left[ 6 \, \Phi_j \, \Phi_{j+1} -4 \, \Phi_j^2 -4 \, \Phi_{j+1}^2 \right] -\left[ 6 \, \Phi_{j-1} \, \Phi_j -4 \, \Phi_{j-1}^2 -4 \, \Phi_j^2 \right] \\
& \, +(\Phi_{j+1}-\Phi_{j-1}) \, (\Phi_{j+2} +\Phi_{j-2}) \\
=& \, T_j -T_{j-1} \, ,
\end{align*}
with
\begin{align*}
T_j \, :=& \, 7 \, \Phi_j \, \Phi_{j+1} -4 \, \Phi_j^2 -4 \, \Phi_{j+1}^2 +\Phi_{j+1} \, \Phi_{j+2} -\Phi_{j-1} \, \Phi_{j+2} +\Phi_{j-1} \, \Phi_j \\
=& \, -4 \, (D \, \Phi_{j+1})^2 +(\Phi_{j+2} -\Phi_j) \, (\Phi_{j+1} -\Phi_{j-1}) \\
=& \, -4 \, (D \, \Phi_{j+1})^2 +(\Delta \, \Phi_{j+1} +2\, D\, \Phi_{j+1}) \, (2\, D\, \Phi_{j+1} -\Delta \, \Phi_j) \\
=& \, -\Delta \, \Phi_j \, \Delta \, \Phi_{j+1} +2\, D \, \Phi_{j+1} \, D \, \Delta \, \Phi_{j+1} 
\, = \, -(D \, \Delta \, \Phi_{j+1} +\Delta \, \Phi_j) \, \Delta \, \Phi_j +2 \, D \, \Phi_{j+1} \, D \, \Delta \, \Phi_{j+1} \, .
\end{align*}
This completes the proof of \eqref{appA-formule3}.
\bigskip

It only remains to prove the formula \eqref{eq:tele_14}, that is:
\begin{equation}
\label{appA-formule4}
\Phi_j \, \Delta^2 \, \Phi_j \, = \, (\Delta \, \Phi_j)^2 + \left[ \Phi_j \, D \, \Delta \, \Phi_{j+1} -D \, \Phi_{j+1} \, \Delta \, \Phi_j \right] 
-\left[ \Phi_{j-1} \, D \, \Delta \, \Phi_j -D \, \Phi_j \, \Delta \, \Phi_{j-1} \right] \, ,
\end{equation}
which is the discrete counterpart of the relation $u \, u'''' =(u'')^2 +(u \, u''' -u' \, u'')'$. We compute:
$$
\Phi_j \, \Delta^2 \, \Phi_j -(\Delta \, \Phi_j)^2 \, = \, \begin{pmatrix}
\Phi_{j+2} \\
\Phi_{j+1} \\
\Phi_j \\
\Phi_{j-1} \\
\Phi_{j-2} \end{pmatrix}^T \, \begin{pmatrix}
0 & 0 & 1/2 & 0 & 0 \\
0 & -1 & 0 & -1 & 0 \\
1/2 & 0 & 2 & 0 & 1/2 \\
0 & -1 & 0 & -1 & 0 \\
0 & 0 & 1/2 & 0 & 0 \end{pmatrix} \, \begin{pmatrix}
\Phi_{j+2} \\
\Phi_{j+1} \\
\Phi_j \\
\Phi_{j-1} \\
\Phi_{j-2} \end{pmatrix} \, ,
$$
and decompose the corresponding symetric matrix in a telescopic way:
$$
\begin{pmatrix}
0 & 0 & 1/2 & 0 & 0 \\
0 & -1 & 0 & -1 & 0 \\
1/2 & 0 & 2 & 0 & 1/2 \\
0 & -1 & 0 & -1 & 0 \\
0 & 0 & 1/2 & 0 & 0 \end{pmatrix} \, = \, \begin{pmatrix}
0 & 0 & 1/2 & 0 & 0 \\
0 & -1 & 0 & -1/2 & 0 \\
1/2 & 0 & 1 & 0 & 0 \\
0 & -1/2 & 0 & 0 & 0 \\
0 & 0 & 0 & 0 & 0 \end{pmatrix} \, - \, \begin{pmatrix}
0 & 0 & 0 & 0 & 0 \\
0 & 0 & 0 & 1/2 & 0 \\
0 & 0 & -1 & 0 & -1/2 \\
0 &1/2 & 0 & 1 & 0 \\
0 & 0 & -1/2 & 0 & 0 \end{pmatrix} \, .
$$
We have thus obtained the telescopic decomposition:
\begin{align*}
\Phi_j \, \Delta^2 \, \Phi_j -(\Delta \, \Phi_j)^2 \, =& \, \left[ \Phi_j \, \Phi_{j+2} -(\Phi_{j+1})^2 +(\Phi_j)^2 -\Phi_{j-1}\, \Phi_{j+1} \right] \\
& \, -\left[ \Phi_{j-1} \, \Phi_{j+1} -(\Phi_j)^2 +(\Phi_{j-1})^2 -\Phi_{j-2}\, \Phi_j \right] \, ,
\end{align*}
and \eqref{appA-formule4} eventually follows from rewriting the telescopic quantity in terms of $\Phi_j$, $D\, \Phi_{j+1}$, $\Delta \, \Phi_j$ and 
$D\, \Delta \, \Phi_{j+1}$:
$$
\Phi_j \, \Phi_{j+2} -(\Phi_{j+1})^2 +(\Phi_j)^2 -\Phi_{j-1}\, \Phi_{j+1} \, = \, \Phi_j \, (\Phi_{j+2} -3 \, \Phi_{j+1} +3 \, \Phi_j -\Phi_{j-1}) 
-(\Phi_{j+1} -\Phi_j) \, (\Phi_{j+1} -2 \, \Phi_j +\Phi_{j-1}) \, . 
$$

\bibliographystyle{alpha}
\bibliography{preprint}

\begin{thebibliography}{BGRSZ02}

\bibitem[Ben14]{Benoit}
A.~Benoit.
\newblock Geometric optics expansions for linear hyperbolic boundary value
  problems and optimality of energy estimates for surface waves.
\newblock {\em Differential Integral Equations}, 27(5-6):531--562, 2014.

\bibitem[BGRSZ02]{BRSZ}
S.~Benzoni-Gavage, F.~Rousset, D.~Serre, and K.~Zumbrun.
\newblock Generic types and transitions in hyperbolic initial-boundary-value
  problems.
\newblock {\em Proc. Roy. Soc. Edinburgh Sect. A}, 132(5):1073--1104, 2002.

\bibitem[CFL28]{cfl}
R.~Courant, K.~Friedrichs, and H.~Lewy.
\newblock \"{U}ber die partiellen {D}ifferenzengleichungen der mathematischen
  {P}hysik.
\newblock {\em Math. Ann.}, 100(1):32--74, 1928.

\bibitem[CG11]{jfcag}
J.-F. Coulombel and A.~Gloria.
\newblock Semigroup stability of finite difference schemes for multidimensional
  hyperbolic initial boundary value problems.
\newblock {\em Math. Comp.}, 80(273):165--203, 2011.

\bibitem[CL20]{CL}
J.-F. Coulombel and F.~Lagouti\`ere.
\newblock The neumann numerical boundary condition for transport equations.
\newblock {\em Kinet. Relat. Models}, 13(1):1--32, 2020.

\bibitem[Cou13]{jfcnotes}
J.-F. Coulombel.
\newblock Stability of finite difference schemes for hyperbolic initial
  boundary value problems.
\newblock In {\em HCDTE Lecture Notes. Part I. Nonlinear Hyperbolic PDEs,
  Dispersive and Transport Equations}, pages 97--225. American Institute of
  Mathematical Sciences, 2013.

\bibitem[GKO95]{gko}
B.~Gustafsson, H.-O. Kreiss, and J.~Oliger.
\newblock {\em Time dependent problems and difference methods}.
\newblock John Wiley \& Sons, 1995.

\bibitem[GKS72]{gks}
B.~Gustafsson, H.-O. Kreiss, and A.~Sundstr{\"o}m.
\newblock Stability theory of difference approximations for mixed initial
  boundary value problems. {II}.
\newblock {\em Math. Comp.}, 26(119):649--686, 1972.

\bibitem[Gol77]{goldberg}
M.~Goldberg.
\newblock On a boundary extrapolation theorem by {K}reiss.
\newblock {\em Math. Comp.}, 31(138):469--477, 1977.

\bibitem[GT78]{goldberg-tadmor1}
M.~Goldberg and E.~Tadmor.
\newblock Scheme-independent stability criteria for difference approximations
  of hyperbolic initial-boundary value problems. {I}.
\newblock {\em Math. Comp.}, 32(144):1097--1107, 1978.

\bibitem[GT81]{goldberg-tadmor}
M.~Goldberg and E.~Tadmor.
\newblock Scheme-independent stability criteria for difference approximations
  of hyperbolic initial-boundary value problems. {II}.
\newblock {\em Math. Comp.}, 36(154):603--626, 1981.

\bibitem[Gus75]{gustafsson}
B.~Gustafsson.
\newblock The convergence rate for difference approximations to mixed initial
  boundary value problems.
\newblock {\em Math. Comp.}, 29(130):396--406, 1975.

\bibitem[Kre66]{kreissproc}
H.-O. Kreiss.
\newblock Difference approximations for hyperbolic differential equations.
\newblock In {\em Numerical {S}olution of {P}artial {D}ifferential {E}quations
  ({P}roc. {S}ympos. {U}niv. {M}aryland, 1965)}, pages 51--58. Academic Press,
  1966.

\bibitem[Kre68]{kreiss1}
H.-O. Kreiss.
\newblock Stability theory for difference approximations of mixed initial
  boundary value problems. {I}.
\newblock {\em Math. Comp.}, 22:703--714, 1968.

\bibitem[RM94]{RM}
R.~D. Richtmyer and K.~W. Morton.
\newblock {\em Difference methods for initial-value problems}.
\newblock Robert E. Krieger Publishing Co., second edition, 1994.

\bibitem[Str62]{strang}
G.~Strang.
\newblock Trigonometric polynomials and difference methods of maximum accuracy.
\newblock {\em J. Math. Phys.}, 41:147--154, 1962.

\bibitem[Wu95]{wu}
L.~Wu.
\newblock The semigroup stability of the difference approximations for
  initial-boundary value problems.
\newblock {\em Math. Comp.}, 64(209):71--88, 1995.

\end{thebibliography}
\end{document}